\documentclass[10pt,twoside,reqno]{amsart}
\usepackage{amsmath, amsthm, amscd, amsfonts, amssymb, graphicx, color}
\usepackage[bookmarksnumbered, colorlinks, plainpages]{hyperref}
\usepackage{mathrsfs}
\usepackage{hhline,booktabs}
\usepackage{multirow}
\usepackage{float}
\usepackage{placeins}
\usepackage{subfigure}
\usepackage{lipsum}
\usepackage{upgreek}
\usepackage{hhline,booktabs}
\usepackage{multirow}
\usepackage{subfigure}
\usepackage[mathscr]{euscript}
\usepackage{multirow,bigstrut,rotating,float}
\usepackage{algorithm,algorithmic}
 \usepackage{kantlipsum}

\numberwithin{equation}{section}
\newtheorem{theorem}{Theorem}[section]
\newtheorem{corollary}[theorem]{Corollary}
\newtheorem{lemma}[theorem]{Lemma}

\newtheorem{definition}[theorem]{Definition}
\newtheorem{remark}[theorem]{\bf{Remark}}
\newtheorem{example}[theorem]{\bf{Example}}

\begin{document}

\title{Efficient iterative techniques for solving tensor problems with the T-product}

\thanks{{\scriptsize
$^\ast$Corresponding author}}
\maketitle

\begin{center}
	
	\textbf{\textbf{Malihe Nobakht Kooshkghazi}}$^{\S}$,\,\,\textbf{\textbf{Salman Ahmadi-Asl}}$^{\ast,\dagger}$,\,\,\textbf{\textbf{Hamidreza Afshin}}$^\S$ \\ [0.2cm]
	$^\S${\small \textit{Department of Mathematics, Vali-e-Asr University of Rafsanjan,\\Rafsanjan, Iran}}\\
	$^\dagger${\small \textit{Lab of Machine Learning and Knowledge Representation, \\ Innopolis University, 420500, Innopolis, Russia.}}\\
	\texttt{e-mails: m.nobakht88@gmail.com,\,salman.ahmadiasl@gmail.com,\,afshin@vru.ac.ir}
\end{center}


\begin{abstract}
{This paper develops two efficient iterative methods for solving tensor equations under the T-product framework. 
For T-symmetric positive definite tensor equations of the form $\mathcal{C} \star \mathcal{X} = \mathcal{D}$, 
we propose a conjugate-gradient-type algorithm that generates orthogonal residual and $\mathcal{C}$-orthogonal 
direction sequences, ensuring convergence within a finite number of steps. For general consistent tensor 
equations, we extend the method using a normal-equation transformation, and further adapt it to handle 
inconsistent systems by solving a least-squares minimization problem. Key advantages include direct 
tensor-based computations without explicit matrix expansion, rigorous finite-step convergence proofs, and the ability to obtain minimal Frobenius norm solutions. Numerical experiments on synthetic data, benchmark images, and video sequences demonstrate that the proposed algorithms achieve high precision with low computational time, confirming their practicality for large-scale multidimensional problems.}

\end{abstract}

\pagestyle{myheadings}
\markboth{\rightline {\scriptsize   
}}
         {\leftline{\scriptsize    }}

\bigskip

\section{Introduction}
A tensor $ \mathscr{C}=(c_{i_{1}\dots i_{d}})_{n_{1}\times \dots \times n_{d}} $ is a multidimensional array with entries $ c_{i_{1} \dots i_{d}} \in \Bbb F$, where $ \Bbb F $ is a field and each index $  i_{j} $ ranges over $1,\dots,n_{j} $ for $ j=1,\dots,d $. The order of $ \mathscr{C} $ is $ d $, and its dimension is the tuple $ (n_{1},\dots,n_{d}) $. We denote this by $ \mathscr{C} \in \Bbb F^{n_{1}\times n_{2}\times \dots \times n_{d}} $. A tensor $  \mathscr{C} $ is called square of order $ d $ and dimension $  n$, If $ n_{1}=n_{2}=\cdots=n_{d}=n $. For more details on tensors, you can see \cite{ten}.

Tensors are applicable in modern problems such as completion \cite{com,Mensah,Asl}, principal component analysis \cite{ana}, Hammerstein identification \cite{Elden}, and image processing \cite{kim2,Ahmadi,LAAP,ACM}. The $n$-mode product has caused many developments on multidimensional data. The CP and Tucker decompositions extend the classical singular value decomposition (SVD) of matrices to tensors \cite{svd}. In the last decade, various product operations have been introduced for two tensors. These include the Einstein product, T-product, C-product, and L-product. Some algorithms utilized for solving matrix equations were generalized to tensor equations using these products \cite{fp1}, \cite{gui2}, \cite{fp}, \cite{sun1}.  

In \cite{kim3} and \cite{kim2}, Kilmer's research group proposed the T-product as extension of matrix multiplication to third-order tensors. The T-product has applications in signal processing, computer vision, image processing and low-rank tensor approximation, for these applications see\cite{sem}, \cite{zhang}, \cite{machine}, \cite{8}, \cite{7}, \cite{9} and \cite{10}. We know that tensor equations including the tensor-tensor product have applications in the real world. Hence, providing methods for solving these equations has always been of interest to researchers, you can see \cite{nobakht3} and references therein. 
For the Einstein product, iterative methods to solve some tensor equations were established in \cite{wen}. Here, we extend these methods to solve equations involving tensor-tensor product.


We structure the remainder as follows: In Section $ 2 $, we present some concepts and notations. Section $ 3 $, presents an algorithm for solving tensor equations of the form $ \mathscr{C} \star \mathscr{X}=\mathscr{D} $, in which tensor $ \mathscr{C} \in \Bbb R^{n \times n \times n_{3}} $ is a symmetric positive definite under tensor-tensot product, $ \mathscr{X} \in \Bbb R^{n \times l \times n_{3}} $ and $ \mathscr{D} \in \Bbb R^{n \times l \times n_{3}} $. In Section $ 4 $, we construct a method to find a solution for the consistent tensor equation $ \mathscr{C} \star \mathscr{X}=\mathscr{D} $, in which $ \mathscr{C} \in \Bbb R^{n_{1} \times n_{2} \times n_{3}} $, $ \mathscr{X} \in \Bbb R^{n_{2} \times l \times n_{3}} $ and $ \mathscr{D} \in \Bbb R^{n_{1} \times l \times n_{3}} $. Also, we analyze the convergence properties of these methods. In Section $ 5 $, we consider the inconsistent tensor equation $ \mathscr{C} \star \mathscr{X}=\mathscr{D} $, in which $ \mathscr{C}, \mathscr{X}  $ and $ \mathscr{D}$ are real tensors with adequate size. Moreover, we present an iterative algorithm for solving the optimization problem $ \min\limits_{\mathscr{X}} \parallel \mathscr{C} \star \mathscr{X}-\mathscr{D} \parallel$, where $\parallel.\parallel$ denote the Frobenius norm of tensors (to be defined in Section $2$). The applicability and efficiency of the proposed approaches are studied in Section $ 6 $. A brief study concludes the paper in Section $7$. 
 
The notations $\mathscr{C}(i, :, :), \mathscr{C}(:, i, :)$ and $\mathscr{C}(:, :, i) $ indicate the $i$th horizontal, lateral and frontal slices of $\mathscr{C}$, respectively. It is common to show the frontal slice $\mathscr{C}(:, :, i)$ by $\mathscr{C}^{(i)}$. The tube is shown by $\mathscr{C}(i, j, :)$. Finally, $ \mathscr{O} $ indicates the zero tensor.

\section{Preliminaries}\label{sec:prelim}

\noindent 
This section provides necessary definitions and summarizes essential T-product properties.

\begin{definition}[\cite{lu}]
The discrete Fourier transform (DFT) matrix $ F_{n} $ has the following form
\begin{equation*}
F_{n}=\begin{bmatrix}
1 & 1 & 1& \ldots &1 \\
1 & \vartheta & \vartheta ^{2}& \ldots & \vartheta^{(n-1)} \\
\vdots & \vdots & \vdots &\ddots & \vdots\\
1& \vartheta^{n-1} &\vartheta^{2(n-1})& \ldots & \vartheta^{(n-1)(n-1)}\\
\end{bmatrix}\in \mathbb{C}^{n\times n},
\end{equation*}
where $\vartheta=e^{-\frac{2\pi i}{n}}$. 
Notice that $\dfrac{F_{n}}{\sqrt{n}}$ is an orthogonal matrix, 
\[
F_{n}^{\ast} F_{n}=F_{n} F_{n}^{\ast}=nI_{n}.
\]
Thus, $ F_{n}^{-1}=\dfrac{F_{n}^{\ast}}{n} $. For a tensor $\mathscr{C} \in \mathbb{R}^{n_{1}\times n_{2} \times n_{3}}$, we define $\overline{\mathscr{C}} \in \mathbb{C}^{n_{1}\times n_{2} \times n_{3}}$ as the tensor obtained by applying the DFT along the third dimension to all tubes of $\mathscr{C}$. By using the Matlab command $\mathtt{fft}$ we obtain
\[
 \overline{\mathscr{C}} = \mathtt{fft}(\mathscr{C}, [ ], 3).
\]
The original tensor $ \mathscr{C} $ can be recovered from $  \overline{\mathscr{C}} $ using the inverse $\mathtt{fft}$:
\[
\mathscr{C} = \mathtt{ifft}(\overline{\mathscr{C}}, [ ], 3).
\]
The block diagonal matrix  $\overline{C} \in \mathbb{C}^{n_{1}n_{3}\times n_{2}n_{3}}$ is defined as
\begin{equation*}
\overline{C} = \mathtt{bdiag}(\overline{\mathscr{C}})=\begin{bmatrix}
\overline{C}^{(1)}& \\
&\overline{C}^{(2)}& \\
 &  &\ddots & \\
&  &&  &\overline{C}^{(n_{3})}\\
\end{bmatrix},
\end{equation*}
in which $ \overline{C}^{(i)}$ is $i\mathrm{th}$ frontal slice of $\overline{\mathscr{C}}$. The block circulant matrix $\mathtt{bcirc}(\mathscr{C}) \in \mathbb{R}^{n_{1}n_{3}\times n_{2}n_{3}}$ is obtained of $ \mathscr{C} $ by
\begin{equation*}
\mathtt{bcirc}(\mathscr{C})=\begin{bmatrix}
C^{(1)}& C^{(n_{3})}&\ldots & C^{(2)}\\
C^{(2)}&C^{(1)}& \ldots & C^{(3)} \\
\vdots &\vdots  &\ddots &\vdots \\
C^{(n_{3})}& C^{(n_{3}-1)}  & \ldots &C^{(1)}\\
\end{bmatrix}.
\end{equation*}
\end{definition}

\begin{definition}[\cite{lu}]
For tensor $ \mathscr{C} $ of order $3$ and dimension $n_{1}\times n_{2}\times n_{3}$, the operations $\mathtt{unfold}$ and $\mathtt{fold}$ are defined as
\begin{equation*}
\mathtt{unfold}(\mathscr{C})=\begin{bmatrix}
C^{(1)}\\
C^{(2)}\\
\vdots \\
C^{(n_{3})}
\end{bmatrix},\ \mathtt{fold}(\mathtt{unfold}(\mathscr{C}))=\mathscr{C},
\end{equation*}
where $\mathtt{unfold}$ maps $ \mathscr{C} $ to a matrix of size $n_{1}n_{3} \times n_{2}$ and $\mathtt{fold}$ is its inverse operator.
\end{definition}
\begin{definition}[\cite{lu}] \label{m11}
Suppose that tensors $\mathscr{C} \in \mathbb{R}^{n_{1}\times n_{2}\times n_{3}}$ and $ \mathscr{D} \in \mathbb{R}^{n_{2}\times l \times n_{3}}$. The $T-product$ $ \mathscr{C}\star \mathscr{D} $ with dimension $ n_{1}\times l \times n_{3} $ is defined as
\[
\mathscr{C} \star \mathscr{D}=\mathtt{fold}(\mathtt{bcirc}(\mathscr{C})\mathtt{unfold}(\mathscr{D})).
\]
Clearly, $\mathscr{E} = \mathscr{C} \star \mathscr{D}$ is equivalent to $\overline{E} =\overline{C}\,\overline{D}$. 
\end{definition}
The Algorithm \ref{algo1n}, presents T-product between tensors $ \mathscr{C} $ and $ \mathscr{D} $.
\begin{algorithm}
   \begin{algorithmic}[1]
    \STATE Input tensors $ \mathscr{C} \in \mathbb{R}^{n_{1}\times n_{2} \times n_{3}} $ and $ \mathscr{D} \in \mathbb{R}^{n_{2}\times l \times n_{3}} $,
     \STATE Compute
       \begin{eqnarray}
       \nonumber \overline{\mathscr{C}} &=& \mathtt{fft}(\mathscr{C}, [ ],3),\\
       \nonumber \overline{\mathscr{D}} &=& \mathtt{fft}(\mathscr{D}, [ ], 3),
       \end{eqnarray}
       \STATE Compute each frontal slice of $\overline{\mathscr{E}}$ by
       \begin{equation*}
       \overline{E}^{(i)} = \begin{cases}
            \overline{C}^{(i)} \overline{D}^{(i)}   & i=1,2,\dots , [\frac{n_{3}+1}{2}],\\
              conj(\overline{E}^{(n_{3}-i+2)})      & i=[\frac{n_{3}+1}{2}]+1,\dots , n_{3},
        \end{cases}
         \end{equation*}
        \STATE Compute $ \mathscr{E}=\mathtt{ifft}(\overline{\mathscr{E}}, [ ], 3) $.
\end{algorithmic} 
	\caption{The Tensor-Tensor Product}
	\label{algo1n}
\end{algorithm} 

Some properties of the T-product are considered in the following lemma.
\begin{lemma}[\cite{jin}]
Let $\mathscr{C},\mathscr{D}$ and $\mathscr{E} $ be tensors of adequate sizes, then the following results hold.
\begin{itemize}
\item
(Left distributivity): $
\mathscr{C} \star (\mathscr{D}+\mathscr{E})=\mathscr{C} \star \mathscr{D}+\mathscr{C} \star \mathscr{E}
$.
\item
(Right distributivity):
$(\mathscr{C} + \mathscr{D})\star \mathscr{E}=\mathscr{C}\star \mathscr{E} + \mathscr{D}\star \mathscr{E}$.
\item
(Associativity): $
\mathscr{C} \star (\mathscr{D} \star \mathscr{E})=(\mathscr{C} \star \mathscr{D}) \star \mathscr{E}$.
\end{itemize}
\end{lemma}
\begin{definition}[\cite{chang}]
Let $ \mathscr{C} $ be a tensor of size $ n_{1}\times n_{2}\times n_{3} $, then $ \mathscr{C}^{T} $ is defined as a tensor of order $3$ and dimension $ n_{2}\times n_{1}\times n_{3} $ obtained by transposing each of the frontal slices and then reversing the order of transposed frontal slices $2$ through $n_{3}$. Also, we have
\[
(\mathscr{C} \star \mathscr{D})^{T}=\mathscr{D}^{T}\star \mathscr{C}^{T}.
\]
\end{definition}
\begin{definition}[\cite{chang}]
The trace of a tensor $\mathscr{C} = (c_{ijk}) \in \mathbb{C}^{m\times m\times n}$ under tensor-tensor product is defined as
\[
Tr(\mathscr{C})=\sum\limits_{i=1}^{m}\sum\limits_{k=1}^{n} c_{iik}.
\]
The following equalities hold for arbitrary tensors $\mathscr{C}, \mathscr{D} \in \mathbb{C}^{ m\times m\times n}$ and constant numbers $c,d$.
\begin{itemize}
\item
$Tr(c \mathscr{C}+d \mathscr{D})= c Tr(\mathscr{C})+d Tr(\mathscr{D}) $.
\item
$Tr(\mathscr{C})=Tr(\mathscr{C}^{T})$.
\item
$ Tr(\mathscr{C}\star \mathscr{D})=Tr(\mathscr{D}\star \mathscr{C})$.
\end{itemize}
\end{definition}
\begin{definition}[\cite{gui}]
The identity tensor $ {\mathscr{I}}_{n_{1} n_{1} n_{3}} $ has matrix $ I_{n_{1}n_{1}} $ as its first frontal slice and zeros elsewhere.
\end{definition}
Clearly, $  \mathscr{C}\star  \mathscr{I}=\mathscr{I}\star  \mathscr{C} =\mathscr{C}$, given the appropriate dimensions. 
\begin{definition}[\cite{gui}]
A tensor $\mathscr{C}$ of dimension  $n_{1} \times n_{1} \times  n_{3}$ is invertible if there exists a tensor $\mathscr{D}$ of dimension $n_{1}\times n_{1} \times n_{3}$ such that
\[
\mathscr{C}\star \mathscr{D}=\mathscr{D}\star \mathscr{C}=\mathscr{I}.
\]
In this context, we take $\mathscr{D} = \mathscr{C}^{-1}$ as the tensor inverse of $ \mathscr{C} $. 
\end{definition}
\begin{definition}[\cite{gui}]
The scalar inner product of $\mathscr{C},\mathscr{D} \in \mathbb{R}^{n_{1} \times n_{2} \times  n_{3}} $ is defined as
\[
< \mathscr{C},\mathscr{D}>=\sum\limits_{k_{1}=1}^{n_{1}} \sum\limits_{k_{2}=1}^{n_{2}}\sum\limits_{k_{3}=1}^{n_{3}} c_{k_{1}k_{2}k_{3}} d_{k_{1}k_{2}k_{3}},
\]
and the Frobenius norm of $ \mathscr{C} $ is given by
\[
{\parallel \mathscr{C} \parallel}= \sqrt{<\mathscr{C} ,\mathscr{C} >}=\sqrt{\sum\limits_{i_{1}=1}^{n_{1}} \sum\limits_{i_{2}=1}^{n_{2}}\sum\limits_{i_{3}=1}^{n_{3}} c_{i_{1}i_{2}i_{3}}^{2}}.
\]
\end{definition}
Also, the scalar product and the associated norm can be computed by
\[
< \mathscr{C},\mathscr{D}>=\frac{1}{n_{3}}< \overline{C}, \overline{D}>, \qquad {\parallel \mathscr{C} \parallel}=\frac{1}{\sqrt{n_{3}}} {\parallel \overline{C} \parallel.}
\]
In the following remark, we compute the inner product of tensors $ \mathscr{C}$ and $ \mathscr{D} $ in terms of the trace of the first frontal slice of $ \mathscr{D}^{T}\star \mathscr{C} $.
\begin{remark}
It is clear that 
\begin{equation}\label{m2}
< \mathscr{C},\mathscr{D}>=\sum\limits_{k_{1},k_{2},k_{3}} c_{k_{1}k_{2}k_{3}} d_{k_{1}k_{2}k_{3}}= Tr((\mathscr{D}^{T} \star \mathscr{C})^{(1)}),
\end{equation}
where $ (\mathscr{D}^{T} \star \mathscr{C})^{(1)} $ is the first frontal slice of $ \mathscr{D}^{T} \star \mathscr{C} $ and $\mathscr{C}, \mathscr{D}$ are real tensors. Tensors $ \mathscr{C}$ and $\mathscr{D} $ are called orthogonal if $ < \mathscr{C},\mathscr{D}>=0 $.
\end{remark}
\begin{definition}[\cite{chang}]
A tensor $ \mathscr{C} \in \mathbb{R}^{n\times n\times n_{3}}$ is called T-symmetric if $\mathscr{C}^{T} = \mathscr{C}$.
\end{definition}
\begin{definition}[\cite{kim3}]
A T-symmetric tensor $\mathscr{C} \in \Bbb R^{n \times n \times n_{3}} $ is T-symmetric positive definite if each $ \overline{C}^{(i)} $ is Hermitian positive definite. 
\end{definition}
It is clear that $ \mathscr{C}^{T}\star \mathscr{C} $ is a T-symmetric semi-positive definite tensor.  
\begin{corollary}
Let $ \mathscr{C} \in \mathbb{R}^{n\times n\times n_{3}}$ be a T-symmetric positive definite, then $Tr[(\mathscr{X}^{T}\star \mathscr{C}\star \mathscr{X})^{(1)}]\geq 0$, for each $ \mathscr{X} \in \Bbb R^{n\times l \times n_{3}} $.
\end{corollary}
\begin{proof}
Since $\mathscr{C}$ is T-symmetric positive definite, the matrix $ \overline{C} $ is symmetric positive definite. Since $ \overline{C} $ is symmetric positive definite, we have
\begin{eqnarray}
\nonumber Tr[(\mathscr{X}^{T} \star\mathscr{C}\star \mathscr{X})^{(1)}]&=& < \mathscr{C}\star \mathscr{X}, \mathscr{X}>\\
\nonumber &=& \frac{1}{n_{3}} < \overline{C} \overline{X}, \overline{X}> \\
\nonumber &=& \frac{1}{n_{3}} Tr(\overline{X}^{T}\overline{C} \overline{X})\geq 0.
\end{eqnarray}
\end{proof}
\begin{definition}[\cite{and}]
The vectorization of a matrix $ C $, written $ vec(C) $, reshapes $ C $ into a single column by concatenating all its columns.
\end{definition}
The important relation between the vec operator and the Kronecker product is
\[
vec(ab^{T})=b\otimes a,
\]
for arbitrary column vectors $a$ and $b$. Also, we have
\[
vec(CXD)=(D^{T} \otimes C) vec(X)
\]
in which, matrices $C$, $D$ and $X$ are of compatible dimensions.

In the following definition, we define some concepts that will be used in the other sections.
\begin{definition}
Suppose that $ \mathscr{C} \in \mathbb{R}^{n\times n\times n_{3}}$ is T-symmetric positive definite, and $\mathscr{R}_{k} , \mathscr{P}_{k} \in \mathbb{R}^{n\times n_{2} \times n_{3}}$, where $ k=1,2,\dots$, are two non-zero tensor sequences. If $ <\mathscr{R}_{s}, \mathscr{R}_{l}>=0 $ and $  <\mathscr{C}\star \mathscr{P}_{s},\mathscr{P}_{l}>=0  $, for all $ s\neq l$ and $s,l=1,2,\dots$, then we call the tensor sequence $ \lbrace \mathscr{R}_{k} \rbrace $ is an orthogonal tensor sequence and the tensor sequence $ \lbrace \mathscr{P}_{k} \rbrace $ is an $\mathscr{C}$-orthogonal tensor sequence.
\end{definition}

\section{An algorithm for T-symmetric positive definite tensor equations}
This section presents a method to solve the equation $ \mathscr{C}\star \mathscr{X} = \mathscr{D}$,
in which $\mathscr{C} \in \mathbb{R}^{n\times n \times n_{3}}  $ is a T-symmetric positive definite tensor, $ \mathscr{X} \in \mathbb{R}^{n\times l \times n_{3}} $ and $ \mathscr{D} \in \mathbb{R}^{n\times l \times n_{3}} $.
\begin{definition}\label{m3}
Let $ f : n\times l \times n_{3} \rightarrow l\times l \times n_{3}$ be a map on tensors, and $ \mathscr{P} \in \mathbb{R}^{n \times l \times n_{3}} $ with ${\parallel \mathscr{P} \parallel}=1$. The derivative of the function $ F(\mathscr{X})=Tr((f(\mathscr{X}))^{(1)}) $ at $ \mathscr{X} $ along the direction $  \mathscr{P} $ is given by
\[
\frac{\partial F}{\partial \mathscr{X}}= \lim_{t \rightarrow 0^{+}} \frac{Tr[(f(\mathscr{X}+t \mathscr{P}))^{(1)}]-Tr[(f(\mathscr{X}))^{(1)}]}{t}.
\]
\end{definition}
\begin{lemma}\label{m6}
Let $\mathscr{C} \in \mathbb{R}^{n\times n \times n_{3}} $ be T-symmetric positive definite, $ \mathscr{D} \in \mathbb{R}^{n\times l \times n_{3}} $, $\mathscr{X} \in \mathbb{R}^{n\times l \times n_{3}}$ and $\mathscr{P} \in \mathbb{R}^{n\times l \times n_{3}}$ with ${\parallel \mathscr{P} \parallel}=1$.
If $ f(\mathscr{X})=\frac{1}{2} \mathscr{X}^{T} \star \mathscr{C}\star \mathscr{X}-\mathscr{X}^{T} \star \mathscr{D} $, then 
\[
\frac{\partial F}{\partial \mathscr{X}}=Tr\bigg[\big((\mathscr{C}\star \mathscr{X}-\mathscr{D})^{T} \star \mathscr{P} \big)^{(1)}\bigg].
\]
\end{lemma}
\begin{proof}
By using the known properties of the transpose of a tensor under the T-product, we can write
\begin{eqnarray}
\nonumber f(\mathscr{X}+t \mathscr{P})&=& \frac{1}{2}(\mathscr{X}+t \mathscr{P})^{T}\star \mathscr{C} \star (\mathscr{X}+t \mathscr{P}) - (\mathscr{X}+t \mathscr{P})^{T}\star \mathscr{D}\\
\nonumber &=& \frac{1}{2} \mathscr{X}^{T}\star \mathscr{C}\star \mathscr{X} + \frac{1}{2}t \big(\mathscr{X}^{T}\star \mathscr{C}\star \mathscr{P} + \mathscr{P}^{T}\star \mathscr{C}\star \mathscr{X}\big)\\
\nonumber &+& \frac{1}{2} t^{2} \mathscr{P}^{T}\star \mathscr{C}\star \mathscr{P} - \mathscr{X}^{T}\star \mathscr{D} - t  \mathscr{P}^{T}\star \mathscr{D}\\
\nonumber &=& f(\mathscr{X}) + \frac{1}{2} t \big( \mathscr{X}^{T}\star \mathscr{C}\star \mathscr{P} + \mathscr{P}^{T}\star \mathscr{C}\star \mathscr{X} \big) - t \mathscr{P}^{T}\star \mathscr{D} + \frac{1}{2} t^{2} \mathscr{P}^{T}\star \mathscr{C}\star \mathscr{P}.
\end{eqnarray}
Thus,
\[
(f(\mathscr{X}+t \mathscr{P}))^{(1)}=\bigg(f(\mathscr{X}) + \frac{1}{2} t \big( \mathscr{X}^{T}\star \mathscr{C}\star \mathscr{P} + \mathscr{P}^{T}\star \mathscr{C}\star \mathscr{X} \big) - t \mathscr{P}^{T}\star \mathscr{D} + \frac{1}{2} t^{2} \mathscr{P}^{T}\star \mathscr{C}\star \mathscr{P}\bigg)^{(1)}.
\]
Definition \ref{m3} allows us to conclude that
\begin{eqnarray}
\nonumber \frac{\partial F}{\partial \mathscr{X}}&=& \lim_{t \rightarrow 0^{+}} \frac{Tr[(f(\mathscr{X}+t \mathscr{P}))^{(1)}]-Tr[(f(\mathscr{X}))^{(1)}]}{t}\\
\nonumber &=& \frac{1}{2} Tr[( \mathscr{X}^{T} \star \mathscr{C}\star \mathscr{P} + \mathscr{P}^{T} \star \mathscr{C} \star \mathscr{X})^{(1)}] - Tr[(\mathscr{P}^{T}\star \mathscr{D})^{(1)}]\\
\nonumber &=& Tr \bigg[\big( (\mathscr{C} \star \mathscr{X}- \mathscr{D})^{T} \star \mathscr{P} \big)^{(1)}\bigg].
\end{eqnarray}
\end{proof}
\begin{lemma}\label{m5}
Let $ \mathscr{C} \in \mathbb{R}^{n \times n \times n_{3}} $ be T-symmetric positive definite, and $ \mathscr{D} \in \mathbb{R}^{n \times l \times n_{3}} $. Then, $ \mathscr{X} \in \mathbb{R}^{n \times l \times n_{3}} $ solves the equation $ \mathscr{C} \star \mathscr{X}=\mathscr{D} $ if and only if $ \mathscr{X} $ solves  the problem
\begin{equation}\label{m4}
\min_{\mathscr{Y}} Tr ((f(\mathscr{Y}))^{1}),
\end{equation}
where $ f(\mathscr{Y}) = \frac{1}{2} \mathscr{Y}^{T}\star \mathscr{C} \star \mathscr{Y}- \mathscr{Y}^{T} \star \mathscr{D}$ and $(f(\mathscr{Y}))^{1}$ is the first frontal slice of $ f(\mathscr{Y}) $.
\end{lemma}
\begin{proof}
First, we suppose that $ \mathscr{X} $ solves the equation $ \mathscr{C} \star \mathscr{X}=\mathscr{D} $. Now, for each $ \mathscr{Y} \in \mathbb{R}^{n \times l \times n_{3}} $,
\begin{eqnarray}
\nonumber f(\mathscr{Y})-f(\mathscr{X})&=& \frac{1}{2} \mathscr{Y}^{T}\star \mathscr{C} \star \mathscr{Y} - \mathscr{Y}^{T}\star \mathscr{D} - \frac{1}{2} \mathscr{X}^{T}\star \mathscr{C} \star \mathscr{X} + \mathscr{X}^{T}\star \mathscr{D}\\
\nonumber &=& \frac{1}{2} (\mathscr{Y}-\mathscr{X})^{T} \star \mathscr{C} \star (\mathscr{Y}-\mathscr{X}) + \frac{1}{2} ( \mathscr{X}^{T}\star \mathscr{C}\star \mathscr{Y}+\mathscr{Y}^{T}\star \mathscr{C }\star \mathscr{X})\\
\nonumber &-& \mathscr{Y}^{T}\star \mathscr{D} - \mathscr{X}^{T}\star \mathscr{C} \star \mathscr{X}+ \mathscr{X}^{T}\star \mathscr{D}\\
\nonumber &=& \frac{1}{2} (\mathscr{Y}-\mathscr{X})^{T} \star \mathscr{C} \star (\mathscr{Y}-\mathscr{X}) + \frac{1}{2}( \mathscr{D}^{T} \star \mathscr{Y} + \mathscr{Y}^{T}\star \mathscr{D})\\
\nonumber &-& \mathscr{Y}^{T}\star \mathscr{D} -  \mathscr{X}^{T}\star \mathscr{D} +  \mathscr{X}^{T}\star \mathscr{D}\\
\nonumber &=& \frac{1}{2} (\mathscr{Y}-\mathscr{X})^{T} \star \mathscr{C} \star (\mathscr{Y}-\mathscr{X})
+ \frac{1}{2}( \mathscr{D}^{T} \star \mathscr{Y} - \mathscr{Y}^{T}\star \mathscr{D}).
\end{eqnarray}
The definition of the transpose {under} T-product concludes that if $\mathscr{S} \in \mathbb{R}^{k \times k\times j}$, then $ Tr((\mathscr{S})^{(1)}) = Tr((\mathscr{S}^{T})^{(1)})$. Also, since $ \mathscr{C} $ is T-symmetric positive definite, 
\begin{eqnarray}
\nonumber Tr[(f(\mathscr{Y})-f(\mathscr{X})^{(1)}]&=& \frac{1}{2} Tr\bigg [\bigg (\frac{1}{2} (\mathscr{Y}-\mathscr{X})^{T} \star \mathscr{C} \star (\mathscr{Y}-\mathscr{X})
+ \frac{1}{2}( \mathscr{D}^{T} \star \mathscr{Y} - \mathscr{Y}^{T}\star \mathscr{D})\bigg)^{(1)}\bigg]\\
\nonumber &=& \frac{1}{2} Tr[((\mathscr{Y}-\mathscr{X})^{T} \star \mathscr{C} \star (\mathscr{Y}-\mathscr{X}))^{(1)}] \geq 0.
\end{eqnarray}
Thus, $ \mathscr{X}$ solves the optimization problem (\ref{m4}).

 On the other hand, let $ \mathscr{E}_{i_{1}i_{2}i_{3}} \in \mathbb{R}^{n\times l \times n_{3}}$ be a tensor with $e_{i_{1}i_{2}i_{3}}=1$ whose other entries are equal to 0. Now, let $ Tr((f(\mathscr{X}))^{(1)})= \min \limits_{\mathscr{Y}} Tr((f(\mathscr{Y}))^{(1)}) $. Then, for each $ t \in \mathbb{R} $ and any tensor $ \mathscr{E}_{i_{1}i_{2}i_{3}} $ we can write
\begin{eqnarray}
\nonumber Tr((f(\mathscr{X}))^{(1)})& \leq & Tr((f(\mathscr{X}+t \mathscr{E}_{i_{1}i_{2}i_{3}}))^{(1)}) \\
\nonumber &=& Tr\bigg[\bigg(f(\mathscr{X})+t \mathscr{E}_{i_{1}i_{2}i_{3}}^{T}\star (\mathscr{C}\star \mathscr{X}- \mathscr{D})+\frac{1}{2} t^{2} \mathscr{E}_{i_{1}i_{2}i_{3}}^{T}\star \mathscr{C}\star  \mathscr{E}_{i_{1}i_{2}i_{3}} \bigg)^{(1)}\bigg].
\end{eqnarray}
Thus,
\[
t Tr[(\mathscr{E}_{i_{1}i_{2}i_{3}}^{T}\star (\mathscr{C}\star \mathscr{X}- \mathscr{D}))^{(1)}]+\frac{1}{2} t^{2} Tr[(\mathscr{E}_{i_{1}i_{2}i_{3}}^{T}\star \mathscr{C}\star  \mathscr{E}_{i_{1}i_{2}i_{3}})^{(1)}] \geq 0.
\]
Therefore, for each $t$,
\[
Tr[(\mathscr{E}_{i_{1}i_{2}i_{3}}^{T}\star (\mathscr{C}\star \mathscr{X}- \mathscr{D}))^{(1)}] = 0.
\]
Thus, $ (\mathscr{C}\star \mathscr{X}- \mathscr{D})_{i_{1}i_{2}i_{3}} =0$ for all $ i_{1},i_{2},i_{3} $, where $ i_{1}=1,2,\dots,n $, $ i_{2}=1,2,\dots,l $ and $ i_{3}=1,2,\dots,n_{3} $. This proves that $ \mathscr{C}\star \mathscr{X}=\mathscr{D} $.
\end{proof}
Now, we use Lemma \ref{m5} to present an approach to solve the equation $ \mathscr{C} \star \mathscr{X}=\mathscr{D} $. First, we select an initial tensor $ \mathscr{X}_{1} $. Then, we must find a direction $ \mathscr{P}_{1} $ along which $ Tr((f(\mathscr{X}))^{(1)}) $ is decreasing. We put $ \mathscr{P}_{1}= \mathscr{R}_{1}=\mathscr{D}- \mathscr{C}\star \mathscr{X}_{1} $. Now, we try to find a tensor $ \mathscr{X}_{2}=\mathscr{X}_{1}+t_{1} \mathscr{P}_{1} $ such that
\[
(f(\mathscr{X}_{1}+t_{1}\mathscr{P}_{1}))^{(1)}= \min\limits_{t> 0} Tr [(f(\mathscr{X}_{1} + t \mathscr{P}_{1}))^{(1)}].
\]
By getting $ g(t)=Tr [(f(\mathscr{X}_{1} + t \mathscr{P}_{1}))^{(1)}] $, we have
\[
g(t)= Tr((f(\mathscr{X}_{1}))^{(1)})-t Tr((\mathscr{R}_{1}^{T}\star \mathscr{P}_{1})^{(1)})+\frac{1}{2} t^{2} Tr((\mathscr{P}_{1}^{T} \star \mathscr{C}\star \mathscr{P}_{1})^{(1)}).
\]
Since the minimum of $ g $ arose at $t_{1}$, $ g^{\prime}(t_{1})=0 $. We put 
\[
g^{\prime}(t)= -Tr((\mathscr{R}_{1}^{T}\star \mathscr{P}_{1} )^{(1)})+t Tr((\mathscr{P}_{1}^{T}\star \mathscr{C}\star \mathscr{P}_{1})^{(1)})=0.
\] 
Thus, we take
\[
t_{1}= \dfrac{Tr((\mathscr{R}_{1}^{T}\star \mathscr{P}_{1})^{(1)})}{Tr((\mathscr{P}_{1}^{T}\star \mathscr{C}\star \mathscr{P}_{1})^{(1)})}=\dfrac{Tr((\mathscr{P}_{1}^{T}\star \mathscr{R}_{1})^{(1)})}{Tr((\mathscr{P}_{1}^{T}\star \mathscr{C}\star \mathscr{P}_{1})^{(1)})}
\]
and we put $ \mathscr{X}_{2} = \mathscr{X}_{1}+t_{1} \mathscr{P}_{1}$, where $ \mathscr{P}_{1}=\mathscr{R}_{1} $. Then,
\[
(f(\mathscr{X}_{1}+t_{1}\mathscr{P}_{1}))^{(1)}= \min\limits_{t> 0} Tr [(f(\mathscr{X}_{1} + t \mathscr{P}_{1}))^{(1)}].
\]
In what follows, we make a tensor $\mathscr{P}_{2} $ orthogonal to $ \mathscr{C}\star \mathscr{P}_{1} $. Suppose that $ \mathscr{P}_{2}=\mathscr{R}_{2}+ \alpha_{1} \mathscr{P}_{1} $, in which $ \mathscr{R}_{2}=\mathscr{D}- \mathscr{C}\star \mathscr{X}_{2}=\mathscr{D}- \mathscr{C}\star (\mathscr{X}_{1}+t_{1} \mathscr{P}_{1})= \mathscr{R}_{1}- t_{1} \mathscr{C}\star \mathscr{P}_{1}$. We know that
\begin{eqnarray}
\nonumber Tr(( \mathscr{P}^{T}_{2}\star \mathscr{C}\star \mathscr{P}_{1})^{(1)})&=& Tr ((\mathscr{R}_{2}^{T} \star \mathscr{C}\star \mathscr{P}_{1})^{(1)}) + \alpha_{1} Tr((\mathscr{P}_{1}^{T}\star \mathscr{C} \star \mathscr{P}_{1})^{(1)})\\
\nonumber &=& Tr(( \mathscr{P}_{1}^{T}\star \mathscr{C}\star \mathscr{R}_{2})^{(1)}) + \alpha_{1} Tr((\mathscr{P}_{1}^{T}\star \mathscr{C} \star \mathscr{P}_{1})^{(1)}).
\end{eqnarray}
Since $ Tr(( \mathscr{P}^{T}_{2}\star \mathscr{C}\star \mathscr{P}_{1})^{(1)})=<\mathscr{C}\star \mathscr{P}_{1}, \mathscr{P}_{2}>=0  $, 
\[
\alpha_{1}=-\frac{Tr(( \mathscr{P}_{1}^{T}\star \mathscr{C}\star \mathscr{R}_{2})^{(1)})}{Tr((\mathscr{P}_{1}^{T}\star \mathscr{C} \star \mathscr{P}_{1})^{(1)})}.
\]
Now, Lemma \ref{m6} allows us to conclude that at $\mathscr{X}=\mathscr{X}_{2}$, $ \frac{\partial F}{\partial \mathscr{X}}$ is 
\begin{eqnarray}
\nonumber \frac{\partial F}{\partial \mathscr{X}}&=& Tr(((\mathscr{C}\star \mathscr{X}_{2}- \mathscr{D})^{T}\star \mathscr{P}_{2})^{(1)})\\
\nonumber &=& - Tr ((\mathscr{R}_{2}^{T}\star (\mathscr{R}_{2}+\alpha_{1}\mathscr{P}_{1}))^{(1)})\\
\nonumber &=& -Tr ((\mathscr{R}_{2}^{T}\star \mathscr{R}_{2})^{(1)})- \alpha_{1} Tr((\mathscr{R}_{2}^{T}\star \mathscr{P}_{1})^{(1)})\\
\nonumber &=& -Tr((\mathscr{R}_{2}^{T}\star \mathscr{R}_{2})^{(1)}) - \alpha_{1} Tr(((\mathscr{R}_{1}-t_{1}\mathscr{C}\star \mathscr{P}_{1})^{T}\star \mathscr{P}_{1})^{(1)})\\
\nonumber &=& -\parallel \mathscr{R}_{2} \parallel^{2} - \alpha_{1} \bigg( Tr((\mathscr{P}_{1}^{T}\star \mathscr{P}_{1})^{(1)}) - \frac{Tr((\mathscr{P}_{1}^{T}\star \mathscr{P}_{1})^{1})}{Tr((\mathscr{P}_{1}^{T}\star \mathscr{C}\star \mathscr{P}_{1})^{(1)})} Tr((\mathscr{P}_{1}^{T}\star \mathscr{C}\star \mathscr{P}_{1})^{(1)}) \bigg)\\
\nonumber &=& - \parallel \mathscr{R}_{2}\parallel^{2}\leq 0.
\end{eqnarray}
Since the derivative of $ F $ at $ \mathscr{X}_{2} $ is less than 0, $ Tr((f(\mathscr{X}))^{(1)}) $ is decreasing along the direction $ \mathscr{P}_{2} $ at $ \mathscr{X}_{2} $.

In what follows, assume that $ \mathscr{X}_{k+1}=\mathscr{X}_{k}+ t_{k} \mathscr{P}_{k}  $, in which $ t_{k} $ solves the problem
 \[
 \min \limits_{t > 0} Tr((f(\mathscr{X}_{k}+t_{k}\mathscr{P}_{k}))^{(1)}),\]
that is,
 \[
  t_{k}=\frac{Tr((\mathscr{P}_{k}^{T}\star \mathscr{R}_{k})^{(1)})}{Tr((\mathscr{P}_{k}^{T}\star \mathscr{C}\star \mathscr{P}_{k})^{(1)})},
\]
where $ \mathscr{R}_{k}=\mathscr{D}-\mathscr{C}\star \mathscr{X}_{k} $. We want to find a tensor $ \mathscr{P}_{k+1} $ orthogonal to $ \mathscr{C}\star \mathscr{P}_{k} $. Let $ \mathscr{P}_{k+1}=\mathscr{R}_{k+1}+\alpha_{k} \mathscr{P}_{k}$, where $ \mathscr{R}_{k+1}= \mathscr{D}- \mathscr{C}\star \mathscr{X}_{k+1}=\mathscr{D}- \mathscr{C}\star (\mathscr{X}_{k}+t_{k}\mathscr{P}_{k})=\mathscr{R}_{k}-t_{k}\mathscr{C}\star \mathscr{P}_{k}$. We obtain
\[
\alpha_{k}=-\frac{Tr((\mathscr{P}_{k}^{T}\star \mathscr{C} \star \mathscr{R}_{k+1})^{1})}{Tr((\mathscr{P}_{k}^{T}\star \mathscr{C}\star \mathscr{P}_{k})^{(1)})}.
\]
Now, we present an algorithm for solving $ \mathscr{C}\star \mathscr{X}=\mathscr{D} $, when $ \mathscr{C} $ is T-symmetric positive definite, in Algorithm \ref{algo1}.
\begin{algorithm}
   \begin{algorithmic}[1]
    \STATE Input tensors $ \mathscr{C} \in \mathbb{R}^{n\times n \times n_{3}} $, $ \mathscr{D} \in \mathbb{R}^{n\times l \times n_{3}} $ and an initial guess $ \mathscr{X}_{1} \in \mathbb{R}^{n\times l \times n_{3}} $.
     \STATE Compute
       \begin{eqnarray}
       \nonumber \mathscr{R}_{1}&=&\mathscr{D}-\mathscr{C}\star \mathscr{X}_{1},\\
      \nonumber \mathscr{P}_{1}&=&\mathscr{R}_{1},\\
      \nonumber k &=&1,
       \end{eqnarray}
      \STATE If ${\| {R}_{k} \|} < tol $, stop.
        \STATE Else, let $k=k+1$,
        \STATE Compute
          \begin{eqnarray}
         \nonumber \mathscr{X}_{k}&=&\mathscr{X}_{k-1}+\frac{Tr((\mathscr{P}_{k-1}^{T}\star         \mathscr{R}_{k-1})^{(1)})}{Tr((\mathscr{P}_{k-1}^{T}\star \mathscr{C}\star \mathscr{P}_{k-1})^{(1)})} \mathscr{P}_{k-1},\\
         \nonumber \mathscr{R}_{k}&=&\mathscr{D}-\mathscr{C}\star \mathscr{X}_{k},\\
         \nonumber  \mathscr{P}_{k}&=& \mathscr{R}_{k}-\frac{Tr((\mathscr{P}_{k-1}^{T}\star         \mathscr{C}\star \mathscr{R}_{k})^{(1)})}{Tr((\mathscr{P}_{k-1}^{T}\star \mathscr{C}\star \mathscr{P}_{k-1})^{(1)})} \mathscr{P}_{k-1},
          \end{eqnarray}
        \STATE Return to step $ 3 $.
\end{algorithmic} 
	\caption{Solving the tensor equation with a T-symmetric positive definite coefficient tensor}
	\label{algo1}
\end{algorithm} 
\begin{theorem}
The tensor sequences $ \lbrace \mathscr{R}_{k}\rbrace_{k=1}^{\infty} $ and $ \lbrace \mathscr{P}_{k}\rbrace_{k=1}^{\infty} $ generated by Algorithm \ref{algo1} are orthogonal tensor and $ \mathscr{C} $-orthogonal tensor sequences, respectively, where $ \mathscr{X}_{1} $ is an arbitrary initial tensor.
\end{theorem}
\begin{proof}
We use induction to prove the theorem. We know that $ \mathscr{P}_{1}=\mathscr{R}_{1} $. Also,
\begin{eqnarray}
\nonumber Tr((\mathscr{R}_{2}^{T}\star \mathscr{R}_{1})^{(1)}) &=& Tr(((\mathscr{R}_{1}-t_{1}\mathscr{C}\star \mathscr{P}_{1})^{T}\star \mathscr{R}_{1})^{(1)})\\
\nonumber &=& Tr((\mathscr{R}_{1}^{T}\star \mathscr{R}_{1} )^{(1)})- t_{1} Tr((\mathscr{P}_{1}^{T}\star \mathscr{C}\star \mathscr{R}_{1})^{(1)})\\
\nonumber &=& Tr((\mathscr{P}_{1}^{T}\star \mathscr{R}_{1})^{(1)}) - \frac{Tr((\mathscr{P}_{1}^{T}\star \mathscr{R}_{1})^{(1)})}{Tr((\mathscr{P}_{1}^{T}\star \mathscr{C}\star \mathscr{P}_{1})^{(1)})}. Tr((\mathscr{P}_{1}^{T}\star \mathscr{C}\star \mathscr{R}_{1})^{(1)})\\
\nonumber &=&0.
\end{eqnarray}
Since $ \mathscr{P}_{2} $ is orthogonal to $ \mathscr{C}\star \mathscr{P}_{1} $, $Tr((\mathscr{P}_{2}^{T} \star \mathscr{C} \star \mathscr{P}_{1} )^{(1)})=0$ (or $<\mathscr{C}\star \mathscr{P}_{1} ,  \mathscr{P}_{2}> =0$). Let $ \lbrace \mathscr{R}_{i}\rbrace_{i=1}^{k} $ and $ \lbrace \mathscr{P}_{i}\rbrace_{i=1}^{k} $ be mutually orthogonal and $ \mathscr{C} $-mutually orthogonal, respectively, for each positive integer $ k $. We must prove that $ <\mathscr{R}_{j} , \mathscr{R}_{k+1}>=0 $ and $< \mathscr{C}\star \mathscr{P}_{j},\mathscr{P}_{k+1}>=0$, for $ j=1,2,\dots,k $.
\begin{eqnarray}
\nonumber Tr((\mathscr{R}_{k+1}^{T}\star \mathscr{R}_{j})^{(1)}) &=& Tr(((\mathscr{R}_{k}- t_{k} \mathscr{C} \star \mathscr{P}_{k})^{T}\star \mathscr{R}_{j})^{(1)})\\
\nonumber &=& Tr((\mathscr{R}_{k}^{T}\star \mathscr{R}_{j})^{(1)})- t_{k} Tr((\mathscr{P}_{k}^{T}\star \mathscr{C}\star \mathscr{R}_{j})^{(1)}) \\
\nonumber &=& Tr((\mathscr{R}_{k}^{T}\star \mathscr{R}_{j})^{(1)}) - t_{k} Tr ((\mathscr{P}_{k}^{T}\star \mathscr{C}\star (\mathscr{P}_{j}-\alpha_{j-1}\mathscr{P}_{j-1}))^{(1)}) \\
\nonumber &=& Tr((\mathscr{R}_{k}^{T}\star \mathscr{R}_{j})^{(1)}) - t_{k} Tr ((\mathscr{P}_{k}^{T}\star \mathscr{C}\star \mathscr{P}_{j})^{(1)}) + t_{k} \alpha_{j-1} Tr ((\mathscr{P}_{k}^{T}\star \mathscr{C}\star \mathscr{P}_{j-1})^{(1)}).
\end{eqnarray}
Since $Tr((\mathscr{R}_{k}^{T}\star \mathscr{R}_{j})^{(1)})=0  $, $ Tr ((\mathscr{P}_{k}^{T}\star \mathscr{C}\star \mathscr{P}_{j})^{(1)})=0 $, and $ Tr ((\mathscr{P}_{k}^{T}\star \mathscr{C}\star \mathscr{P}_{j-1})^{(1)}) =0 $, in which $ j=1,2,\dots,k-1 $, we have
\[
Tr((\mathscr{R}_{k+1}^{T}\star \mathscr{R}_{j})^{(1)})=0, \qquad j=1,2,\dots,k-1.
\]
It is clear that 
\begin{eqnarray}\label{m8}
\nonumber \mathscr{P}_{k}&=&\mathscr{R}_{k}+\alpha_{k-1} \mathscr{P}_{k-1}\\
\nonumber &=& \mathscr{R}_{k}+\alpha_{k-1} \mathscr{R}_{k-1}+\alpha_{k-1}\alpha_{k-2} \mathscr{P}_{k-2}\\
 &=& \mathscr{R}_{k}+\alpha_{k-1} \mathscr{R}_{k-1}+\alpha_{k-1}\alpha_{k-2} \mathscr{P}_{k-2} + \dots + \alpha_{k-1}\alpha_{k-2} \dots \alpha_{1} \mathscr{R}_{1}.
\end{eqnarray}
Now, relation (\ref{m8}) and induction concludes that
\begin{eqnarray}\label{m9}
Tr((\mathscr{P}_{k}^{T}\star \mathscr{R}_{k})^{(1)})= Tr((\mathscr{R}_{k}^{T}\star \mathscr{R}_{k})^{(1)}).
\end{eqnarray}
Thus, by using relation (\ref{m9}) we obtain
\begin{eqnarray}
\nonumber Tr((\mathscr{R}_{k+1}^{T}\star \mathscr{R}_{k})^{(1)}) &=& Tr((\mathscr{R}_{k}^{T}\star \mathscr{R}_{k})^{(1)})- t_{k} Tr((\mathscr{P}_{k}^{T}\star \mathscr{C}\star \mathscr{P}_{k})^{(1)})\\
\nonumber &+& t_{k}\alpha_{k-1} Tr((\mathscr{P}_{k}^{T}\star \mathscr{C}\star \mathscr{P}_{k-1})^{(1)})\\
\nonumber &=& Tr((\mathscr{R}_{k}^{T}\star \mathscr{R}_{k})^{(1)}) - t_{k} Tr((\mathscr{P}_{k}^{T}\star \mathscr{C}\star \mathscr{P}_{k})^{(1)})\\
\nonumber &=& Tr((\mathscr{R}_{k}^{T}\star \mathscr{R}_{k})^{(1)}) - \frac{Tr((\mathscr{P}_{k}^{T}\star \mathscr{R}_{k})^{(1)})}{Tr((\mathscr{P}_{k}^{T}\star \mathscr{C}\star \mathscr{P}_{k})^{(1)})} Tr((\mathscr{P}_{k}^{T}\star \mathscr{C}\star \mathscr{P}_{k})^{(1)})\\
\nonumber &=&0.
\end{eqnarray}
It is clear that $ Tr((\mathscr{P}_{k+1} \star \mathscr{A}\star \mathscr{P}_{k})^{(1)}) =0$. Also,
\begin{eqnarray}
\nonumber Tr((\mathscr{P}_{k+1}^{T} \star \mathscr{C}\star \mathscr{P}_{j})^{(1)}) &=& Tr((\mathscr{R}_{k+1}+\alpha_{k} \mathscr{P}_{k})^{T} \star \mathscr{C} \star \mathscr{P}_{j})^{(1)})\\
\nonumber &=& Tr((\mathscr{R}_{k+1}^{T}\star \mathscr{C}\star \mathscr{P}_{j})^{(1)}) + \alpha_{k} Tr((\mathscr{P}_{k}^{T}\star \mathscr{C}\star \mathscr{P}_{j})^{(1)})\\
\nonumber &=& Tr((\mathscr{R}_{k+1}^{T}\star \frac{1}{t_{j}}( \mathscr{R}_{j}- \mathscr{R}_{j+1}))^{(1)})\\
\nonumber &=&0,
\end{eqnarray}
where $ j=1,2,\dots,k-1 $. Hence, the proof is complete by induction.
\end{proof}
By relation (\ref{m9}), the $ \mathscr{X}_{k}$ in step $ 5 $ of Algorithm \ref{algo1} can be expressed as
\begin{eqnarray}
\nonumber \mathscr{X}_{k}&=&\mathscr{X}_{k-1}+\frac{Tr((\mathscr{P}_{k-1}^{T}\star         \mathscr{R}_{k-1})^{(1)})}{Tr((\mathscr{P}_{k-1}^{T}\star \mathscr{C}\star \mathscr{P}_{k-1})^{(1)})} \mathscr{P}_{k-1}\\
\nonumber &=& \mathscr{X}_{k-1}+\frac{Tr((\mathscr{R}_{k-1}^{T}\star \mathscr{R}_{k-1})^{(1)})}{Tr((\mathscr{P}_{k-1}^{T}\star \mathscr{C}\star \mathscr{P}_{k-1})^{(1)})} \mathscr{P}_{k-1}\\
\nonumber &=& \mathscr{X}_{k-1}+\frac{<\mathscr{R}_{k-1},\mathscr{R}_{k-1}>}{Tr((\mathscr{P}_{k-1}^{T}\star \mathscr{C}\star \mathscr{P}_{k-1})^{(1)})} \mathscr{P}_{k-1}\\
\nonumber &=& \mathscr{X}_{k-1}+\frac{\parallel \mathscr{R}_{k-1}\parallel^{2}}{Tr((\mathscr{P}_{k-1}^{T}\star \mathscr{C}\star \mathscr{P}_{k-1})^{(1)})} \mathscr{P}_{k-1}.
\end{eqnarray}
\section{An algorithm for solving consistent tensor equations}
This section presents a technique to solve the consistent tensor equation $ \mathscr{C}\star \mathscr{X}=\mathscr{D} $. The Definition \ref{m11} concludes that the consistent equation $ \mathscr{C}\star \mathscr{X}=\mathscr{D} $ is equivalent to the tensor equation $ \mathscr{C}^{T}\star \mathscr{C}\star \mathscr{X}=\mathscr{C}^{T}\star\mathscr{D} $, in which $ \mathscr{C} \in \mathbb{R}^{n_{1}\times n_{2}\times n_{3}} $, $ \mathscr{X} \in \mathbb{R}^{n_{2}\times l \times n_{3}} $ and $ \mathscr{D} \in \mathbb{R}^{n_{1}\times l \times n_{3}} $.

 Since $ \mathscr{C}^{T}\star \mathscr{C} $ is T-symmetric positive definite, we present the following algorithm for solving consistent tensor equations by modifying Algorithm \ref{algo1}.
\begin{algorithm}
   \begin{algorithmic}[1]
    \STATE Input tensors $ \mathscr{C} \in \mathbb{R}^{n_{1}\times n_{2}\times n_{3}} $, $ \mathscr{D} \in \mathbb{R}^{n_{1}\times l \times n_{3}}  $ and an initial guess $ \mathscr{X}_{1} \in \mathbb{R}^{n_{2}\times l \times n_{3}} $,
     \STATE Compute
       \begin{eqnarray}
       \nonumber \mathscr{R}_{1}&=&\mathscr{D}-\mathscr{C}\star \mathscr{X}_{1},\\
      \nonumber \mathscr{P}_{1}&=& \mathscr{C}^{T}\star \mathscr{R}_{1},\\
      \nonumber \mathscr{Q}_{1}&=& \mathscr{P}_{1},\\
      \nonumber k &=&1,
       \end{eqnarray}
      \STATE If ${\| {R}_{k} \|} < tol $ or ${\| {R}_{k} \|} \geq tol $ but $ {\| {Q}_{k} \|} < tol $, stop.
        \STATE Else, let $k=k+1$,
        \STATE Compute
          \begin{eqnarray}
         \nonumber \mathscr{X}_{k}&=&\mathscr{X}_{k-1}+ \frac{\parallel \mathscr{R}_{k-1} \parallel^{2}}{\parallel \mathscr{Q}_{k-1} \parallel^{2}} \mathscr{Q}_{k-1},\\
         \nonumber \mathscr{R}_{k}&=&\mathscr{D}-\mathscr{C}\star\mathscr{X}_{k},\\
         \nonumber  \mathscr{P}_{k}&=& \mathscr{C}^{T}\star\mathscr{R}_{k},\\
         \nonumber  \mathscr{Q}_{k}&=&\mathscr{P}_{k}- \frac{Tr((\mathscr{P}_{k}^{T}\star \mathscr{Q}_{k-1} )^{(1)})}{\parallel \mathscr{Q}_{k-1} \parallel^{2}} \mathscr{Q}_{k-1},
          \end{eqnarray}
        \STATE Return to step $ 3 $.
\end{algorithmic} 
	\caption{Solving the consistent tensor equation $ \mathscr{C}\star \mathscr{X}=\mathscr{D} $}
	\label{algo2}
\end{algorithm}  
\begin{lemma}\label{m15}
The tensor sequences $ \lbrace \mathscr{R}_{i} \rbrace$ and $ \lbrace \mathscr{Q}_{i} \rbrace$ generated in Algorithm \ref{algo2} satisfy the relation
\[
<\mathscr{R}_{j} , \mathscr{R}_{i+1}>=<\mathscr{R}_{j} , \mathscr{R}_{i}> - \frac{\parallel \mathscr{R}_{i} \parallel^{2}}{\parallel \mathscr{Q}_{i} \parallel^{2}} <\mathscr{P}_{j} ,\mathscr{Q}_{i}>, \qquad i,j=1,2,\dots
\]
for any initial tensor $ \mathscr{X}_{1} $.
\end{lemma}
\begin{proof}
By relation (\ref{m2}) and Algorithm \ref{algo2}, 
\begin{eqnarray}
\nonumber Tr((\mathscr{R}_{i+1}^{T}\star \mathscr{R}_{j})^{(1)}) &=& Tr(((\mathscr{D}-\mathscr{C}\star \mathscr{X}_{i+1})^{T}\star \mathscr{R}_{j})^{(1)})\\
\nonumber &=& Tr((\mathscr{D}-\mathscr{C}\star (\mathscr{X}_{i}- \frac{\parallel \mathscr{R}_{i} \parallel^{2}}{\parallel \mathscr{Q}_{i} \parallel^{2}} \mathscr{Q}_{i})^{T}\star \mathscr{R}_{j})^{(1)})\\
\nonumber &=& Tr \big[(  (\mathscr{D}-\mathscr{C}\star \mathscr{X}_{i})^{T}\star \mathscr{R}_{j}- \frac{\parallel \mathscr{R}_{i} \parallel^{2}}{\parallel \mathscr{Q}_{i} \parallel^{2}} (\mathscr{C}\star \mathscr{Q}_{i})^{T}  \star \mathscr{R}_{j})^{(1)}\big]\\
\nonumber &=&  Tr((\mathscr{R}_{i}^{T}\star \mathscr{R}_{j})^{(1)})- \frac{\parallel \mathscr{R}_{i} \parallel^{2}}{\parallel \mathscr{Q}_{i} \parallel^{2}} Tr((\mathscr{Q}_{i}^{T}\star \mathscr{C}^{T}\star \mathscr{R}_{j})^{(1)})\\
\nonumber &=&  Tr((\mathscr{R}_{i}^{T}\star \mathscr{R}_{j} )^{(1)})- \frac{\parallel \mathscr{R}_{i} \parallel^{2}}{\parallel \mathscr{Q}_{i} \parallel^{2}} Tr((\mathscr{Q}_{i}^{T}\star \mathscr{P}_{j})^{(1)}).
\end{eqnarray}
\end{proof}
The below lemma proves that if $ \mathscr{Q}_{i}=\mathscr{O} $, then $ \mathscr{R}_{i}=\mathscr{O} $.
\begin{lemma}\label{m19}
Let $ \mathscr{X}^{*} $ solves the consistent tensor equation $\mathscr{C}\star \mathscr{X}=\mathscr{D} $, where $ \mathscr{C} \in \mathbb{R}^{n_{1}\times n_{2}\times n_{3}} $, $ \mathscr{X} \in \mathbb{R}^{n_{2}\times l \times n_{3}} $ and $ \mathscr{D} \in \mathbb{R}^{n_{1}\times l \times n_{3}} $. Then, the tensor sequences $ \lbrace \mathscr{R}_{i}\rbrace $ and $ \lbrace \mathscr{Q}_{i}\rbrace $ generated by Algorithm \ref{algo2} satisfy
\begin{eqnarray}\label{m12}
Tr[((\mathscr{X}^{*}-\mathscr{X}_{k})\star \mathscr{Q}_{k}^{T})^{(1)}]= \parallel \mathscr{R}_{k}\parallel^{2}, \qquad k=1,2,\dots
\end{eqnarray}
\end{lemma}
\begin{proof}
By induction on $ k $, we prove the lemma. Let $ k=1 $. Then,
\begin{eqnarray}
\nonumber Tr[((\mathscr{X}^{*}-\mathscr{X}_{1})\star \mathscr{Q}_{1}^{T})^{(1)}] &=& Tr[((\mathscr{X}^{*}-\mathscr{X}_{1})\star \mathscr{P}_{1}^{T})^{(1)}]\\
\nonumber &=& Tr[((\mathscr{X}^{*}-\mathscr{X}_{1})\star \mathscr{R}_{1}^{T} \star \mathscr{C})^{(1)}]\\
\nonumber &=& Tr[(\mathscr{R}_{1}^{T} \star  \mathscr{C} \star (\mathscr{X}^{*} - \mathscr{X}_{1}))                    ^{(1)}]\\
\nonumber &=& Tr[(\mathscr{R}_{1}^{T}\star (\mathscr{D}-\mathscr{C} \star \mathscr{X}_{1}))^{(1)}]\\
\nonumber &=& Tr ((\mathscr{R}_{1}^{T}\star \mathscr{R}_{1})^{(1)})\\
\nonumber &=& \parallel \mathscr{R}_{1} \parallel^{2}.
\end{eqnarray}
If equality (\ref{m12}) is true for $ k=1,2,\dots,n $, then 
\begin{eqnarray}\label{m13}
\nonumber Tr (((\mathscr{X}^{*}- \mathscr{X}_{n+1})\star \mathscr{Q}_{n}^{T})^{(1)}) &=& Tr \bigg[ (\mathscr{X}^{*} - \mathscr{X}_{n} - \frac{\parallel \mathscr{R}_{n} \parallel^{2}}{\parallel \mathscr{Q}_{n} \parallel^{2}}\mathscr{Q}_{n})\star \mathscr{Q}_{n}^{T})^{(1)} \bigg]\\
\nonumber &=& Tr [(\mathscr{X}^{*} - \mathscr{X}_{n})\star \mathscr{Q}_{n}^{T} )^{(1)}] - \frac{\parallel \mathscr{R}_{n} \parallel^{2}}{\parallel \mathscr{Q}_{n} \parallel^{2}} Tr((\mathscr{Q}_{n}\star \mathscr{Q}_{n}^{T})^{(1)}) \\
\nonumber &=& \parallel \mathscr{R}_{n} \parallel^{2}- \parallel \mathscr{R}_{n} \parallel^{2}\\
 &=&0.
\end{eqnarray}
Now, by relation (\ref{m13}) we obtain
\begin{eqnarray}
\nonumber Tr (((\mathscr{X}^{*}- \mathscr{X}_{n+1})\star \mathscr{Q}_{n+1}^{T})^{(1)}) &=& Tr \bigg [
((\mathscr{X}^{*}- \mathscr{X}_{n+1})\star (\mathscr{P}_{n+1}-\frac{Tr(( \mathscr{P}_{n+1}^{T}\star  \mathscr{Q}_{n})^{(1)})}{\parallel  \mathscr{Q}_{n} \parallel^{2}}\mathscr{Q}_{n})^{T} )^{(1)}\bigg]\\
\nonumber &=& Tr[ ((\mathscr{X}^{*}- \mathscr{X}_{n+1})\star \mathscr{P}_{n+1}^{T})^{(1)}]\\
\nonumber &=& Tr[ ((\mathscr{X}^{*}- \mathscr{X}_{n+1})\star \mathscr{R}_{n+1}^{T} \star \mathscr{C})^{(1)}]\\
\nonumber &=& Tr[ ( \mathscr{R}_{n+1}^{T} \star \mathscr{C}\star (\mathscr{X}^{*}- \mathscr{X}_{n+1}))^{(1)}]\\
\nonumber &=& Tr[ ( \mathscr{R}_{n+1}^{T} \star(\mathscr{D}- \mathscr{C}\star \mathscr{X}_{n+1}))^{(1)}]\\
\nonumber &=& Tr((\mathscr{R}_{n+1}^{T} \star \mathscr{R}_{n+1})^{(1)})\\
\nonumber &=& \parallel \mathscr{R}_{n+1} \parallel^{2}.
\end{eqnarray}
\end{proof}
The following lemma proves that the tensor sequence $ \lbrace \mathscr{R}_{i} \rbrace $ generated in Algorithm 
\ref{algo2} is orthogonal. Also, the same is true for $ \lbrace \mathscr{Q}_{i} \rbrace $.
\begin{lemma}\label{m20}
Let two tensor sequences $ \lbrace \mathscr{R}_{i} \rbrace $ and $ \lbrace \mathscr{Q}_{i} \rbrace $ be generated in Algorithm \ref{algo2}, where $ \mathscr{R}_{i}\neq \mathscr{O} $ and $i=1,\dots,k$. Then,
\begin{eqnarray}\label{m14}
< \mathscr{R}_{j},\mathscr{R}_{i}> =0, \qquad < \mathscr{Q}_{j},\mathscr{Q}_{i}> =0, \quad \forall i,j=1,\dots,k, \quad i \neq j.
\end{eqnarray}
\end{lemma}
\begin{proof}
We use induction on $ k $ to prove this lemma. Let $ k=2 $. Then, by using Lemma \ref{m15},
\begin{eqnarray}
\nonumber Tr((\mathscr{R}_{2}^{T}\star \mathscr{R}_{1})^{(1)})&=& Tr((\mathscr{R}_{1}^{T}\star \mathscr{R}_{1})^{(1)})- \frac{\parallel \mathscr{R}_{1} \parallel^{2}}{\parallel \mathscr{Q}_{1} \parallel^{2}} Tr((\mathscr{Q}_{1}^{T} \star \mathscr{P}_{1})^{(1)})\\
\nonumber &=& \parallel \mathscr{R}_{1} \parallel^{2} - \frac{\parallel \mathscr{R}_{1} \parallel^{2}}{\parallel \mathscr{Q}_{1} \parallel^{2}} Tr((\mathscr{Q}_{1}^{T} \star \mathscr{Q}_{1})^{(1)})\\
\nonumber &=&0.
\end{eqnarray}
Also,
\begin{eqnarray}
\nonumber Tr((\mathscr{Q}_{2}^{T}\star \mathscr{Q}_{1})^{(1)})&=& Tr[((\mathscr{P}_{2}- \frac{Tr((\mathscr{P}_{2}^{T} \star \mathscr{Q}_{1})^{(1)})}{\parallel \mathscr{Q}_{1} \parallel^{2}} \mathscr{Q}_{1})^{T}\star \mathscr{Q}_{1})^{(1)}]\\
\nonumber &=& Tr ((\mathscr{P}_{2}^{T}\star \mathscr{Q}_{1})^{(1)}) - \frac{Tr((\mathscr{P}_{2}^{T} \star \mathscr{Q}_{1})^{(1)})}{\parallel \mathscr{Q}_{1} \parallel^{2}} Tr((\mathscr{Q}_{1}^{T} \star \mathscr{Q}_{1})^{(1)})\\
\nonumber &=& 0.
\end{eqnarray}
If relation (\ref{m14}) holds for $ k=1,2,\dots,n $, then
\begin{eqnarray}
\nonumber Tr((\mathscr{R}_{n+1}^{T}\star \mathscr{R}_{j})^{(1)}) &=& Tr((\mathscr{R}_{n}^{T}\star \mathscr{R}_{j})^{(1)}) - \frac{\parallel \mathscr{R}_{n} \parallel^{2}}{\parallel \mathscr{Q}_{n} \parallel^{2}} Tr((\mathscr{Q}_{n}^{T}\star \mathscr{P}_{j})^{(1)})\\
\nonumber &=& Tr((\mathscr{R}_{n}^{T}\star \mathscr{R}_{j})^{(1)}) - \frac{\parallel \mathscr{R}_{n} \parallel^{2}}{\parallel \mathscr{Q}_{n} \parallel^{2}} Tr[(\mathscr{Q}_{n}^{T} \star (\mathscr{Q}_{j}+ \frac{Tr((\mathscr{P}_{j}^{T} \star \mathscr{Q}_{j-1})^{(1)})}{\parallel \mathscr{Q}_{j-1} \parallel^{2}}     \mathscr{Q}_{j-1}))^{(1)}]\\
\nonumber &=&0,
\end{eqnarray}
where $ j=1,2,\dots,n-1 $. Now, let $ j=n $. Then,
\begin{eqnarray}
\nonumber Tr((\mathscr{R}_{n+1}^{T}\star \mathscr{R}_{n})^{(1)}) &=& Tr((\mathscr{R}_{n}^{T}\star \mathscr{R}_{n})^{(1)}) - \frac{\parallel \mathscr{R}_{n} \parallel^{2}}{\parallel \mathscr{Q}_{n} \parallel^{2}} Tr((\mathscr{Q}_{n}^{T}\star \mathscr{P}_{n})^{(1)})\\
\nonumber &=& \parallel \mathscr{R}_{n} \parallel^{2} - \frac{\parallel \mathscr{R}_{n} \parallel^{2}}{\parallel \mathscr{Q}_{n} \parallel^{2}} Tr\bigg[( \mathscr{Q}_{n}^{T} \star (\mathscr{Q}_{n}+ \frac{Tr((\mathscr{P}_{n}^{T}\star \mathscr{Q}_{n-1})^{(1)})}{\parallel \mathscr{Q}_{n-1} \parallel^{2}}         \mathscr{Q}_{n-1}) )^{(1)}\bigg]\\
\nonumber &=&0.
\end{eqnarray}
In what follows, we consider the sequence $ \lbrace \mathscr{Q}_{i} \rbrace$ and prove its orthogonality. We can write
\begin{eqnarray}\label{m17}
\nonumber Tr((\mathscr{Q}_{n+1}^{T}\star \mathscr{Q}_{j})^{(1)}) &=& Tr [((\mathscr{P}_{n+1} - \frac{Tr((\mathscr{P}_{n+1}^{T}\star \mathscr{Q}_{n})^{(1)}) }{\parallel \mathscr{Q}_{n} \parallel^{2}}\mathscr{Q}_{n})^{T}\star \mathscr{Q}_{j})^{(1)}]\\
 &=& Tr((\mathscr{P}_{n+1}^{T}\star \mathscr{Q}_{j})^{(1)}) - \frac{Tr((\mathscr{P}_{n+1}^{T}\star \mathscr{Q}_{n})^{(1)})}{\parallel \mathscr{Q}_{n} \parallel^{2}} Tr((\mathscr{Q}_{n}^{T}\star \mathscr{Q}_{j})^{(1)}).
\end{eqnarray}
If we put $ j=n $, then equation (\ref{m17}) becomes zero. If $j=1,2,\dots,n-1$, then Lemma \ref{m15} allows us to conclude that
\begin{eqnarray}
\nonumber Tr((\mathscr{Q}_{n+1}^{T}\star \mathscr{Q}_{j})^{(1)})&=& Tr((\mathscr{P}_{n+1}^{T}\star \mathscr{Q}_{j})^{(1)})\\
\nonumber &=&Tr((\mathscr{Q}_{j}^{T}\star \mathscr{P}_{n+1})^{(1)})\\
\nonumber &=&\frac{\parallel \mathscr{Q}_{j} \parallel^{2}}{\parallel \mathscr{R}_{j} \parallel^{2}} \bigg[ Tr((\mathscr{R}_{j}^{T}\star \mathscr{R}_{n+1})^{(1)}) - Tr((\mathscr{R}_{j+1}^{T}\star \mathscr{R}_{n+1})^{(1)})\bigg]\\
\nonumber &=& 0.
\end{eqnarray}
\end{proof}
The following theorem finds the greatest number of steps for the convergence of Algorithm \ref{algo2}.
\begin{theorem}\label{m21}
Under the notations in Algorithm \ref{algo2}, the sequence $ \lbrace \mathscr{X}_{k} \rbrace $ produced by Algorithm \ref{algo2} converges to an exact solution of the consistent tensor equation $ \mathscr{C}\star \mathscr{X}=\mathscr{D} $ in at most $ n_{1}l n_{3}+1 $ iteration steps for any initial tensor $\mathscr{X}_{1}  \in \mathbb{R}^{n_{2}\times l \times n_{3}} $.
\end{theorem}
\begin{proof}
Let $\mathscr{R}_{i}$, where $i=1,2,\dots,n_{1}ln_{3}$, be a non-zero tensor. The Lemma \ref{m19} shows that $\mathscr{Q}_{i}\neq \mathscr{O}  $, in which $i=1,2,\dots,n_{1}ln_{3}$. Thus, Algorithm \ref{algo2} produces $\mathscr{X}_{n_{1}ln_{3}+1}, \mathscr{R}_{n_{1}ln_{3}+1} $. Now, Lemma \ref{m20} tells us that sequence
$ \lbrace \mathscr{R}_{1},\mathscr{R}_{2},\dots, \mathscr{R}_{n_{1}ln_{3}+1}\rbrace $ is orthogonal  in $ \mathbb{R}^{n_{1}\times l\times n_{3}} $. But, this is impossible, since the space $ \lbrace \mathscr{X}\vert \mathscr{X} \in \mathbb{R}^{n_{1}\times l\times n_{3}} \rbrace $ has dimension $n_{1} l n_{3}$. Hence, $ \mathscr{R}_{n_{1}ln_{3}+1}= \mathscr{O} $ and $ \mathscr{X}_{n_{1}ln_{3}+1} $ is the exact solution.
\end{proof}
\begin{corollary}
The tensor equation $ \mathscr{C}\star \mathscr{X}=\mathscr{D} $ is inconsistent if and only if there exists an integer $ k $ for which $ \mathscr{R}_{k} \neq \mathscr{O} $ while $ \mathscr{Q}_{k} =\mathscr{O} $.
\end{corollary}
\begin{proof}
The sufficiency of this corollary can be proved by using Lemma \ref{m19}. Now, let the equation $ \mathscr{C}\star \mathscr{X}=\mathscr{D} $ be inconsistent. Then, $ \mathscr{R}_{i} \neq 0$ for all $ i $. Now, let $ \mathscr{Q}_{i} \neq 0$ for all $ i $. Now, Theorem \ref{m21} concludes that the equation $ \mathscr{C}\star \mathscr{X}=\mathscr{D} $ has a solution. But, this contradicts the assumption.
\end{proof}
 We know that, for $ C \in \mathbb{R}^{m\times n} $ and $ d \in \mathbb{R}^{m} $ with consistent equation $ Cx=d $. The solution $ x^{*}=C^{T} y$, where $ y \in \mathbb{R}^{m} $, satisfies:
\[
{\parallel x^{*}\parallel}\leq {\parallel x\parallel}, \quad \forall  x \in \lbrace x \in \Bbb R^{n}\vert Cx=d\rbrace,
\]
with equality if and only if $x=x^{*}$. The following theorem presents a minimal Frobenius norm solution of $ \mathscr{C}\star \mathscr{X}=\mathscr{D} $.
\begin{theorem}\label{m40}
Under the notations in Algorithm \ref{algo2}, consider $ \mathscr{X}_{1}=\mathscr{C}^{T}\star \mathscr{H} $,  where $ \mathscr{H} \in \mathbb{R}^{n_{1}\times l \times n_{3}} $ is arbitrary, then the solution $\tilde{\mathscr{X}} $ obtained by Algorithm \ref{algo2} has minimal Frobenius norm. In particular, we can put $ \mathscr{X}_{1}=\mathscr{O} \in \mathbb{R}^{n_{2}\times l \times n_{3}} $.
\end{theorem}
\begin{proof}
In Algorithm \ref{algo2}, if we put the tensor $ \mathscr{X}_{1}=\mathscr{C}^{T}\star \mathscr{H} $, where $ \mathscr{H} \in \mathbb{R}^{n_{1}\times l \times n_{3}} $ is an arbitrary tensor, then Algorithm \ref{algo2} products the solution $ \tilde{\mathscr{X}} $ of the form
\begin{eqnarray}\label{m25}
\tilde{\mathscr{X}}=\mathscr{C}^{T}\star \mathscr{Y}, \quad \mathscr{Y} \in \mathbb{R}^{n_{1}\times l \times n_{3}}.
\end{eqnarray}
By Definition \ref{m11}, we can write 
\begin{eqnarray}\label{m24}
\nonumber \mathscr{C}\star \mathscr{X}=\mathscr{D} &\Leftrightarrow & \overline{C}\overline{X}=\overline{D}\\
 &\Leftrightarrow & (I_{l n_{3}} \otimes \overline{C}) vec(\overline{X}) =vec(\overline{D}),
\end{eqnarray}
in which $ C=(c_{1},c_{2},\dots,c_{n}) \in \mathbb{R}^{m \times n} $ implies $ vec(C)=(c_{1}^{T},c_{2}^{T},\dots , c_{n}^{T})^{T} \in \mathbb{R}^{m n\times 1}$. 
By relation (\ref{m24}), if $ \mathscr{X}^{*} $ is the solution of $ \mathscr{C}\star \mathscr{X}=\mathscr{D} $ that satisfies the relation
\[
vec(\overline{X^{*}}) \in \lbrace (I_{l n_{3}}\otimes \overline{C})^{T}y, \quad y \in \mathbb{R}^{n_{1}n_{3} l  n_{3}\times 1 } \rbrace,
\]
then $ \mathscr{X}^{*} $ is unique and has minimal Frobenius norm. We use (\ref{m25}) and conclude that
\begin{eqnarray}\label{m26}
\nonumber \tilde{\mathscr{X}}=\mathscr{C}^{T}\star \mathscr{Y} &\Leftrightarrow &\overline{\tilde{X}}= \overline{C^{T}} \overline{Y}\\
&\Leftrightarrow & vec(\overline{\tilde{X}})=(I_{l n_{3}}\otimes \overline{C^{T}}) vec(\overline{Y}).
\end{eqnarray}
Thus,
\begin{eqnarray}
\nonumber  vec(\overline{\tilde{X}}) \in \lbrace (I_{l n_{3}}\otimes \overline{C})^{T}y, \quad y\in \mathbb{R}^{n_{1} n_{3} l n_{3}\times 1} \rbrace,
\end{eqnarray}
and this completes the proof.
\end{proof}
\section{An algorithm for solving inconsistent tensor equations}
The solutions $ \tilde{\mathscr{X}} $ of the optimization problem $ \min\limits_{\mathscr{X}}  \parallel \mathscr{C}\star \mathscr{X} -\mathscr{D} \parallel$ are called the least-square solutions of the tensor equation $\mathscr{C}\star \mathscr{X} =\mathscr{D} $. In the following, we consider the inconsistent equation $ \mathscr{C}\star \mathscr{X} =\mathscr{D} $ and compose an iterative approach to solve the problem $ \min\limits_{\mathscr{X}}  \parallel \mathscr{C}\star \mathscr{X} -\mathscr{D} \parallel$, in which $ \mathscr{C}$, $ \mathscr{X}$ and $ \mathscr{D} $ are tensors with edequate size.

\begin{lemma}\label{m32}
If $ \tilde{\mathscr{X}} $ is the minimizer of 
\begin{eqnarray}\label{m30}
\varphi (\mathscr{X}) = \parallel \mathscr{C}\star \mathscr{X} - \mathscr{D} \parallel^{2}, 
\end{eqnarray}
where $ \mathscr{C} \in \mathbb{R}^{n_{1}\times n_{2}\times n_{3}} $, $ \mathscr{X} \in \mathbb{R}^{n_{2}\times l\times n_{3}} $ and $ \mathscr{D} \in \mathbb{R}^{n_{1}\times l \times n_{3}} $, then $ \tilde{\mathscr{X}} $ satisfies
\begin{eqnarray}\label{m31}
\mathscr{C}^{T} \star \mathscr{C} \star \mathscr{X}= \mathscr{C}^{T} \star \mathscr{D}.
\end{eqnarray}
\end{lemma}
\begin{proof}
 Let $ \mathscr{E}_{i_{1}i_{2}i_{3}} \in \mathbb{R}^{n_{2}\times l \times n_{3}}$ be a tensor with $e_{i_{1}i_{2}i_{3}}=1$ whose other entries are equal to 0. Now, let $ \parallel \mathscr{C}\star \mathscr{X} - \mathscr{D} \parallel^{2}= \min \limits_{\mathscr{Y}} \parallel \mathscr{C}\star \mathscr{Y} - \mathscr{D} \parallel^{2}$. Then, for each $ t \in \mathbb{R} $ and any tensor $ \mathscr{E}_{i_{1}i_{2}i_{3}} $ we can write
\begin{eqnarray}
\nonumber \varphi (\mathscr{X}) \leq \varphi (\mathscr{X}+t \mathscr{E}_{i_{1}i_{2}i_{3}}) & \Leftrightarrow & Tr [((\mathscr{X}^{T} \star \mathscr{C}^{T}) \star (\mathscr{C} \star \mathscr{X}) - 2 \mathscr{D}^{T}\star (\mathscr{C} \star \mathscr{X}) + 2 \mathscr{D}^{T}\star \mathscr{D})^{(1)}] \\
\nonumber & \leq & Tr [(((\mathscr{X}+t \mathscr{E}_{i_{1}i_{2}i_{3}})^{T} \star \mathscr{C}^{T}) \star (\mathscr{C} \star (\mathscr{X}+t \mathscr{E}_{i_{1}i_{2}i_{3}})) \\
\nonumber &- &2 \mathscr{D}^{T}\star (\mathscr{C} \star (\mathscr{X}+t \mathscr{E}_{i_{1}i_{2}i_{3}})) + 2 \mathscr{D}^{T} \star \mathscr{D})^{(1)}] \\
\nonumber &=& Tr[( \mathscr{X}^{T} \star \mathscr{C}^{T}\star \mathscr{C} \star \mathscr{X} - 2 \mathscr{D}^{T}\star \mathscr{C} \star \mathscr{X} + 2 \mathscr{D}^{T}\star \mathscr{D} +  t  (\mathscr{X}^{T} \star \mathscr{C}^{T}\star \mathscr{C}) \star \mathscr{E}_{i_{1}i_{2}i_{3}}\\
\nonumber & + & t \mathscr{E}_{i_{1}i_{2}i_{3}}^{T}\star (\mathscr{C}^{T}\star \mathscr{C} \star \mathscr{X}) + t^{2}  \mathscr{E}_{i_{1}i_{2}i_{3}}^{T} \star (\mathscr{C}^{T}\star \mathscr{C}) \star             \mathscr{E}_{i_{1}i_{2}i_{3}} \\
\nonumber &-& 2 t (\mathscr{D}^{T}\star \mathscr{C}) \star \mathscr{E}_{i_{1}i_{2}i_{3}})^{(1)}].
\end{eqnarray}
Then,
\begin{eqnarray}
\nonumber Tr[( t  (\mathscr{X}^{T} \star \mathscr{C}^{T}\star \mathscr{C}) \star \mathscr{E}_{i_{1}i_{2}i_{3}} &+& t \mathscr{E}_{i_{1}i_{2}i_{3}}^{T}\star (\mathscr{C}^{T}\star \mathscr{C} \star \mathscr{X}) + t^{2}  \mathscr{E}_{i_{1}i_{2}i_{3}}^{T} \star (\mathscr{C}^{T}\star \mathscr{C}) \star  \mathscr{E}_{i_{1}i_{2}i_{3}}\\
\nonumber & - & 2 t (\mathscr{D}^{T}\star \mathscr{C}) \star \mathscr{E}_{i_{1}i_{2}i_{3}})^{(1)}]\geq 0.
\end{eqnarray}
Thus, for each $ t $,
\begin{eqnarray}
\nonumber t \bigg[Tr(2\mathscr{E}_{i_{1}i_{2}i_{3}}^{T}\star (\mathscr{C}^{T}\star \mathscr{C} \star \mathscr{X}) + t  \mathscr{E}_{i_{1}i_{2}i_{3}}^{T} \star (\mathscr{C}^{T}\star \mathscr{C}) \star  \mathscr{E}_{i_{1}i_{2}i_{3}} -  2 \mathscr{E}_{i_{1}i_{2}i_{3}}^{T} \star (\mathscr{C}^{T}\star \mathscr{D}))^{(1)}\bigg] \geq 0.
\end{eqnarray}
Hence,
\begin{eqnarray}
\nonumber Tr[(2\mathscr{E}_{i_{1}i_{2}i_{3}}^{T}\star (\mathscr{C}^{T}\star \mathscr{C} \star \mathscr{X}- \mathscr{C}^{T}\star \mathscr{D}))^{(1)}]=0.
\end{eqnarray}
Thus, for every $1\leq i_{1} \leq n_{2} $, $ 1\leq i_{2} \leq l $ and $ 1\leq i_{3} \leq n_{3} $, $ \mathscr{E}_{i_{1}i_{2}i_{3}}^{T}\star (\mathscr{C}^{T}\star \mathscr{C} \star \mathscr{X}- \mathscr{C}^{T}\star \mathscr{D})_{i_{1}i_{2}i_{3}} =0 $, and this shows that $ \mathscr{C}^{T} \star \mathscr{C} \star \mathscr{X}= \mathscr{C}^{T} \star \mathscr{D} $.
 \end{proof}
\begin{remark}\label{m35}
By Lemma \ref{m32}, the least-square solution of $ \min\limits_{\mathscr{X}} \parallel \mathscr{C}\star \mathscr{X}-\mathscr{D} \parallel$, where $ \mathscr{C}\star \mathscr{X}=\mathscr{D} $ is inconsistent, is the solution set of the consistent normal equation (\ref{m31}).  
\end{remark}
In Algorithm \ref{algo3}, we present an approach for solving $ \min\limits_{\mathscr{X}} \parallel \mathscr{C}\star \mathscr{X}-\mathscr{D} \parallel$, where $ \mathscr{C}\star \mathscr{X}=\mathscr{D} $ is inconsistent, by using Algorithm \ref{algo2}.
\begin{algorithm}
   \begin{algorithmic}[1]
    \STATE Input tensors $ \mathscr{C} \in \mathbb{R}^{n_{1}\times n_{2}\times n_{3}} $, $ \mathscr{D} \in \mathbb{R}^{n_{1}\times l \times n_{3}}  $ and an initial guess $ \mathscr{X}_{1} \in \mathbb{R}^{n_{2}\times l \times n_{3}} $.
     \STATE Compute
       \begin{eqnarray}
       \nonumber \mathscr{R}_{1}&=&\mathscr{C}^{T}\star \mathscr{D}-\mathscr{C}^{T}\star\mathscr{C}\star \mathscr{X}_{1},\\
      \nonumber \mathscr{P}_{1}&=& \mathscr{C}^{T}\star \mathscr{C}\star \mathscr{R}_{1},\\
      \nonumber \mathscr{Q}_{1}&=& \mathscr{P}_{1},\\
      \nonumber k &=&1,
       \end{eqnarray}
      \STATE If ${\| {R}_{k} \|} < tol $, then stop.
        \STATE Else, let $k=k+1$,
        \STATE Compute
          \begin{eqnarray}
         \nonumber \mathscr{X}_{k}&=&\mathscr{X}_{k-1} + \frac{\parallel \mathscr{R}_{k-1} \parallel^{2}}{\parallel \mathscr{Q}_{k-1} \parallel^{2}} \mathscr{Q}_{k-1},\\
         \nonumber \mathscr{R}_{k}&=&\mathscr{C}^{T}\star \mathscr{D} - \mathscr{C}^{T}\star \mathscr{C} \star\mathscr{X}_{k},\\
         \nonumber  \mathscr{P}_{k}&=& \mathscr{C}^{T}\star \mathscr{C} \star\mathscr{R}_{k},\\
         \nonumber  \mathscr{Q}_{k}&=&\mathscr{P}_{k}- \frac{Tr((\mathscr{P}_{k}^{T}\star \mathscr{Q}_{k-1}  )^{(1)})}{\parallel \mathscr{Q}_{k-1} \parallel^{2}} \mathscr{Q}_{k-1},
          \end{eqnarray}
        \STATE Return to step $ 3 $.
\end{algorithmic} 
	\caption{Solving the least-square problem $ \min\limits_{\mathscr{X}} \parallel \mathscr{C}\star \mathscr{X}-\mathscr{D} \parallel$}
	\label{algo3}
\end{algorithm} 
\begin{remark}
It is clear that each solution of the consistent tensor equation $ \mathscr{C}\star \mathscr{X}=\mathscr{D} $ is a solution of $ \min\limits_{\mathscr{X}} \parallel \mathscr{C}\star \mathscr{X}-\mathscr{D} \parallel$. Thus, we can use Algorithm \ref{algo3} for solving the consistent equation $ \mathscr{C}\star \mathscr{X}=\mathscr{D} $.
\end{remark}
\begin{theorem}
Under the notations in Algorithm \ref{algo3}, the sequence $ \lbrace \mathscr{X}_{k} \rbrace$ generated by Algorithm \ref{algo3} converges to a solution $ \min\limits_{\mathscr{X}} \parallel \mathscr{C}\star \mathscr{X}-\mathscr{D} \parallel$ in at most $ n_{2} l  n_{3} $ iteration steps, where $ \mathscr{X}_{1} \in \mathbb{R}^{n_{2}\times l \times n_{3}} $ is an arbitrary tensor. In addition, if we put the tensor $ \mathscr{X}_{1}= \mathscr{C}^{T}\star \mathscr{C}\star \mathscr{H}$, then the solution $ \tilde{\mathscr{X}} $ generated by Algorithm \ref{algo3} is the minimal Frobenius norm solution of $ \min\limits_{\mathscr{X}} \parallel \mathscr{C}\star \mathscr{X}-\mathscr{D} \parallel$, where $ \mathscr{H} \in \mathbb{R}^{n_{2}\times l \times n_{3}}  $ is arbitrary. Also, we can put $ \mathscr{X}_{1}=\mathscr{O} \in \mathbb{R}^{n_{2}\times l \times n_{3}} $.
\end{theorem} 
\begin{proof}
The proof of this theorem is the same as the proofs of Theorems \ref{m21} and \ref{m40}. 
\end{proof}

\begin{remark}\label{rem:com}
{Please note that the general method for solving $\mathcal{C} \star \mathcal{X} = \mathcal{D}$ can be implemented in a direct way, namely via the tubal inverse: $\mathcal{X} = \mathcal{C}^{-1} \star \mathcal{D}$ (or $\mathcal{X} = \mathcal{C}^{\dagger} \star \mathcal{D}$ when $\mathcal{C}$ is singular). This approach requires inverting $n_3$ matrices of size $n_1 \times n_1$ after an FFT. It becomes impractical for large $n_1$ or $n_3$, especially when $\mathcal{C}$ has structure (e.g., sparsity) that cannot be exploited. Iterative methods like the ones proposed here are preferable: they avoid explicit inversion, have lower computational complexity ($O(k n_1 n_2 n_3 \log n_3)$ versus $O(n_3 n_1^3)$), require less memory, and can obtain minimal-norm solutions. To our knowledge, no other iterative algorithms have been developed for T-product tensor equations, making direct comparison infeasible. }
\end{remark}

%
%

\section{Numerical examples}
In this section, we show the effectiveness of our proposed algorithms through comprehensive numerical tests. All computations were performed using MATLAB R2014a on a Windows 7 system with 3 GB RAM. The PSNR of two images ${\mathscr X}$ and ${\mathscr Y}$ is defined as
\[
{\rm PSNR}=10\log _{10}\left({\frac{255^2}{{\rm MSE}}}\right)\,{\rm dB},
\]
where ${{\rm MSE}=\frac{\parallel{\mathscr X}-{\mathscr Y}\parallel^2}{{\rm num}({\mathscr X})}}$ and ''num(${\mathscr X}$)'' represent the total number of elements in data tensor ${\mathscr X}$. The residual norm is  ${\parallel\mathscr{C}\star \mathscr{X}-\mathscr{D}\parallel}$ and the relative error is also defined as ${\frac{\parallel\mathscr{C}\star \mathscr{X}-\mathscr{D}\parallel}{\parallel\mathscr{D}\parallel}}$.
\begin{example}
Assume that $n_{1}=5, n_{2}=4, n_{3}=3, l=5$. We randomly constructed a tensor $ \mathscr{C} $ by the MATLAB function randn.m:
\begin{equation*}
\mathscr{C}(:,:,1)=\begin{bmatrix}
    1.7380 & -10.6399  &  1.4411  &  0.4655\\
   -0.9092  & -5.8846  & -7.9709  & -1.8908\\
   -4.6977  & -4.9527  &  0.5511 &  -7.4134\\
   -0.1877  & -5.8652  &  3.9353  & -0.2191\\
   -9.4815 &  -8.6271  & -0.0111   & 4.8041
\end{bmatrix}
\end{equation*}
\begin{equation*}
\mathscr{C}(:,:,2)=\begin{bmatrix}
    8.6912 &  -1.1348  & -6.6081  &  3.8850\\
   -2.1510 &  -5.7446 &  -3.1806  &  3.1120\\
   -8.1366 &  10.1217 &   1.5893  &  3.2369\\
    0.8317 & -11.7976  &  0.6902  & -2.1282\\
    1.8813 &  -2.5499 &  -3.5537  &  5.2429
\end{bmatrix}
\end{equation*}
\begin{equation*}
\mathscr{C}(:,:,3)=\begin{bmatrix}
   3.3035  & -6.4419  & -2.7839  & -4.7632\\
   12.5439  & -1.8561 &  -4.4756 &   1.5866\\
    5.3173  & -3.7890  & -2.0466  &  0.3901\\
    5.7846  & -2.8198 &  -0.8044  &  6.6219\\
    0.2649   & 2.7757  &  2.0467 &  -1.0659
\end{bmatrix}.
\end{equation*}
We select the tensor $ \mathscr{D}= \mathscr{C}\star \mathscr{X}^{*}$, where $\mathscr{X}^{*} =ones(n_{2},l,n_{3}) $, and applied Algorithm \ref{algo2} to the equation $ \mathscr{C}\star \mathscr{X} = \mathscr{D} $ to compute $ \mathscr{X}_{k} $ with the initial tensor $ \mathscr{X}_{1}= \mathscr{O} $. 
The elapsed time was $0.005767$ seconds and the residual norm was $2.0186e-12$, for $k=5$.

As shown in Figure \ref{fig1} below, the residual norm converges to zero when the iteration count $ k $ increases. 

\begin{figure}
	\centering
	\includegraphics[width=0.8\linewidth]{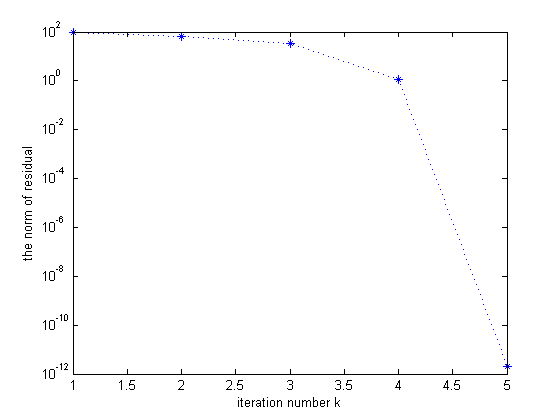}
	\caption{The norm of residual for example 1.}\label{fig1}
\end{figure}

As highlighted in Remark~\ref{rem:com}, the direct method can also be used to solve the tubal equation. Here, we examine this approach and compare it with our proposed method. To this end, we include a small experimental comparison between our method and the direct tubal inverse for a modest-sized problem. Table~\ref{tab:empirical} shows preliminary results on a randomly generated tensor with $n_1 = 50$, $n_2 = 40$, $n_3 = 30$, and $l = 20$.

\begin{table}[htbp]
\centering
\caption{Empirical comparison on synthetic data ($n_1=50$, $n_2=40$, $n_3=30$, $l=20$)}
\label{tab:empirical}
\begin{tabular}{c|c|c}
\toprule
\textbf{Method} & \textbf{Time (seconds)} & \textbf{Memory (MB)} \\
\midrule
Direct tubal inversion & 2.34 & 48.5 \\
Proposed iterative (tol=$10^{-10}$) & 1.01 & 12.8 \\
\bottomrule
\end{tabular}
\end{table}

These results empirically support our complexity analysis: the iterative method achieves a speedup of approximately $2.31\times$ while using only $26\%$ of the memory required by direct inversion.

\end{example}

\begin{example}
Assume that $n_{1}=5, n_{2}=4, n_{3}=3, l=3$. We randomly constructed a tensor $ \mathscr{C} $ by the MATLAB function randn.m:
\begin{equation*}
\mathscr{C}(:,:,1)=\begin{bmatrix}
    3.3077 &  -2.4998  & -2.9964  &  1.3519\\
   10.6925 &  1.9151 &  -2.9479 &  -3.2639\\
    2.7057  &  2.0602 &   4.2677 &   2.3861\\
   -7.7044 &   2.0275  & -9.2650 &  -0.3566\\
   -1.0157 &  -1.8189 &  -1.0365 &  -4.6915
\end{bmatrix},
\end{equation*}
\begin{equation*}
\mathscr{C}(:,:,2)=\begin{bmatrix}
   0.8068 &  -9.2306  & -3.6714  &  3.1976 \\
   -1.3409 &  -1.9917 &   2.7032  & -0.4049\\
   -2.0494 &  -2.7177 &   4.8792 &   2.7044 \\
   -3.5566 &  -4.5595 &  -0.7844  & -6.3128 \\
    0.3072 &   3.2635  &  1.3890  &  5.5521
\end{bmatrix},
\end{equation*}
\begin{equation*}
\mathscr{C}(:,:,3)=\begin{bmatrix}
   -4.9478 &  -1.8163 &  -0.9867  &  2.9893\\
   -9.1442  & -5.1029  &  2.0280 &  -6.4064\\
    6.9225 & -15.3649 &  -7.0967 & -11.0163 \\
   -0.3136  &  3.1314 &  -3.6472 &  -2.8562 \\
    2.2446 &  -1.4334  &  5.7366  &  1.0700 \\
\end{bmatrix}.
\end{equation*}
Also, we randomly constructed a tensor $ \mathscr{D} $ by the MATLAB function randn.m:
\begin{equation*}
\mathscr{D}(:,:,1)=\begin{bmatrix}
    0.9424 &  -0.9610  & -0.2857\\
    0.0937  & -0.6537 &  -0.4624\\
   -1.1223 &  -1.2294 &  -0.4098\\
    0.3062 &  -0.2710  & -0.5035\\
   -1.1723  & -0.9000 &   1.2333
\end{bmatrix},
\end{equation*}
\begin{equation*}
\mathscr{D}(:,:,2)=\begin{bmatrix}
   0.6103  &  2.6052  &  0.5476\\
    0.0591 &   0.9724 &   1.5651\\
   -1.4669  &  0.2570  & -1.6933\\
   -1.6258 &  -0.9742  & -0.4494\\
   -1.9648 &  -1.1464 &  -0.0843
\end{bmatrix},
\end{equation*}
\begin{equation*}
\mathscr{D}(:,:,3)=\begin{bmatrix}
    -1.9920 &   0.4092  &  1.3018\\
    0.8412  & -1.1424  & -0.5936\\
   -0.4147 &  -0.6249 &   0.4364\\
    1.9122 &  -1.1687  & -0.5044\\
   -0.3909 &   0.3926 &   0.1021\\
\end{bmatrix}.
\end{equation*}
We applied Algorithm \ref{algo3} to obtain the minimal Frobenius norm approximate solution of the problem $ \min\limits_{\mathscr{X}} \parallel \mathscr{C} \star \mathscr{X}- \mathscr{D} \parallel$ with initial tensor $ \mathscr{X}_{1}= \mathscr{O} $ as follows.
\begin{equation*}
\mathscr{X}(:,:,1)=\begin{bmatrix}
   0.1322 &   0.1079  & -0.1833\\
    0.0133  & -0.1052 &  -0.0152\\
   -0.1267  & -0.0997  &  0.0924\\
    0.0200  &  0.2322 &  -0.0438
\end{bmatrix},
\end{equation*}
\begin{equation*}
\mathscr{X}(:,:,2)=\begin{bmatrix}
   0.1314  &  0.1165 &  -0.0877\\
    0.1348 &   0.0194 &  -0.1298\\
    0.0217  & -0.0115 &   0.1025\\
   -0.1365 &   0.0623  &  0.1802
\end{bmatrix},
\end{equation*}
\begin{equation*}
\mathscr{X}(:,:,3)=\begin{bmatrix}
    0.0541  &  0.0006  & -0.2426\\
    0.0820   & 0.1508  & -0.0388\\
   -0.2393  &  -0.0697 &   0.1875\\
    0.0562 &  -0.1318  &  0.0374
\end{bmatrix}.
\end{equation*}
 The elapsed time was $0.018662$ seconds, and the residual norm was $1.0352e-08$, for $k=10$.
 
As shown in Figure \ref{fig2}, the residual norm decays to zero with increasing $k$. 
\begin{figure}
	\centering
	\includegraphics[width=0.8\linewidth]{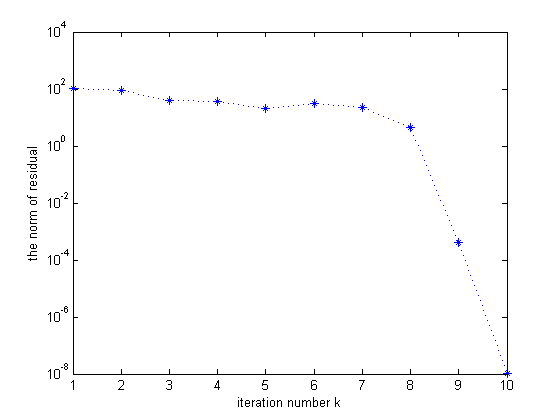}
	\caption{The relative error for example 2.}\label{fig2}
\end{figure}
\end{example}

\begin{example}
In the last two experiments, we worked on synthetic data tensors. In this example, using real-world data, we show that the presented algorithm not only works with random data but also with realistic data tensors. To do so, we consider benchmark ``peppers'', ``baboon'', ``airplane'', and ``house'' images depicted in Figure \ref{sample_fig}. We consider the standard Gaussian random tensor (the Gaussian tensor with zero mean and unit variance) as the coefficient tensor $\mathscr{C}$ and degrade the image using the model $\mathscr{X}_{\rm deg}=\mathscr{C}\star \mathscr{X}_{\rm orig}$, where $\mathscr{X}_{\rm orig}$ and $\mathscr{X}_{\rm deg}$ are the original and degraded images. For the benchmark images mentioned above, the degraded images and the ones recovered by solving the tensor lest-squares problem
\[{\min\limits_{\mathscr{X}}\,\parallel \mathscr{C}\star\mathscr{X}-\mathscr{X}_{\rm deg}\parallel},
\]
are reported in Figure \ref{fig_recoverd}. Also, the relative error of the proposed algorithm for four benchmark images are shown in Figure \ref{Err_fig}. The PSNR of the recovered images achieved by the proposed algorithm were quite high, 150.4637 for ``peppers'', 145.2125 for ``house'', 140.7686 for ``airplane'' and 145.6455 for ``baboon''. As can be seen, the restored images have quite high quality in terms of PSND. 

In the next set of experiments, we examine the proposed algorithm with video data sets. Here, we consider ``News'' and ``Foreman'' videos from \url{https://media.xiph.org/video/derf/}. The size of these videos is $144\times 176\times 300$ (300 frames of size $144\times 176$) and for simplicity of computations, we only consider 30 frames and so consider a tensor of size $144\times 176\times 30$ in our computations. A few frames of both videos are shown in Figure \ref{video_sample}. Similar to the image case, the original video is degraded with a random Gaussian tensor (zero mean and unit variance) and the proposed iterative algorithm is applied to extract the original video. The relative error history for both videos are displayed in Figure \ref{Err_fig_video}. Also, the recovered frame number 1 for these two videos are shown in Figure \ref{video_recoverd}. These experiments clearly show the efficiency of the proposed methodology for not only solving linear tensor equations but also for real-world applications. 


\begin{figure}
	\centering
	\includegraphics[width=0.5\linewidth]{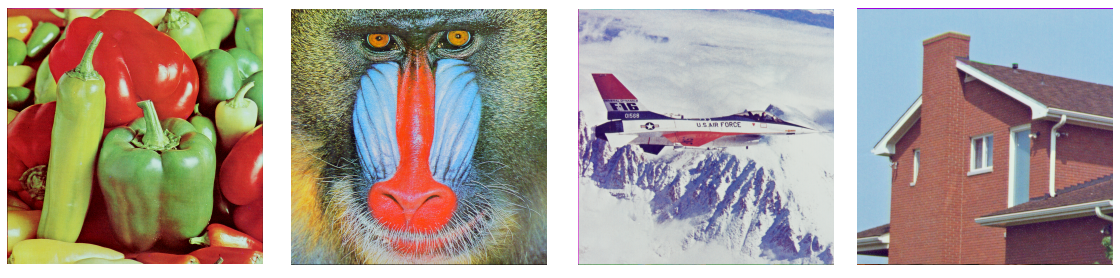}
	\caption{\small Benchmark images used in our simulations.}\label{sample_fig}
\end{figure}

\begin{figure}[H]
	\centering
	\includegraphics[width=0.45\linewidth]{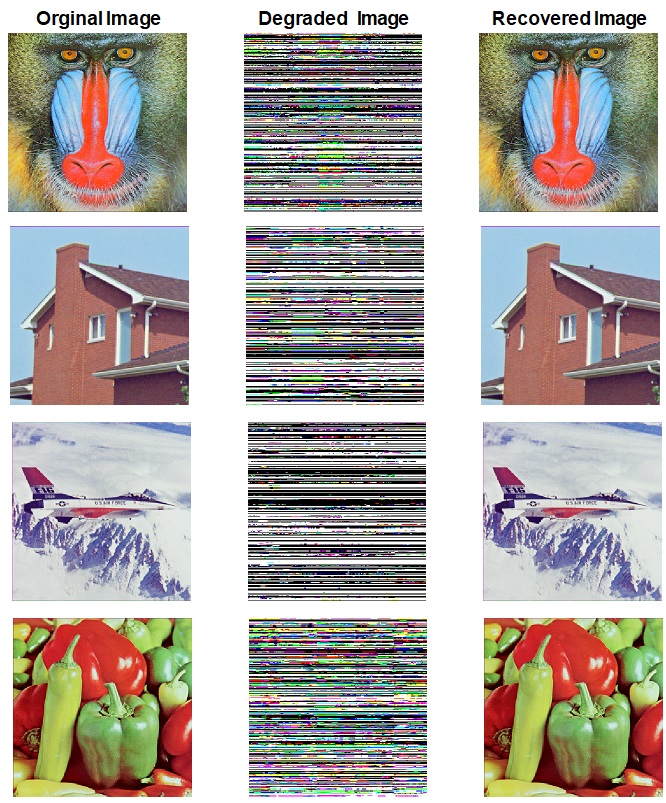}
	\caption{\small The original, the degraded and the recovered images using the proposed iterative method.}\label{fig_recoverd}
\end{figure}

\begin{figure}[H]
	\centering
	\includegraphics[width=0.9\linewidth]{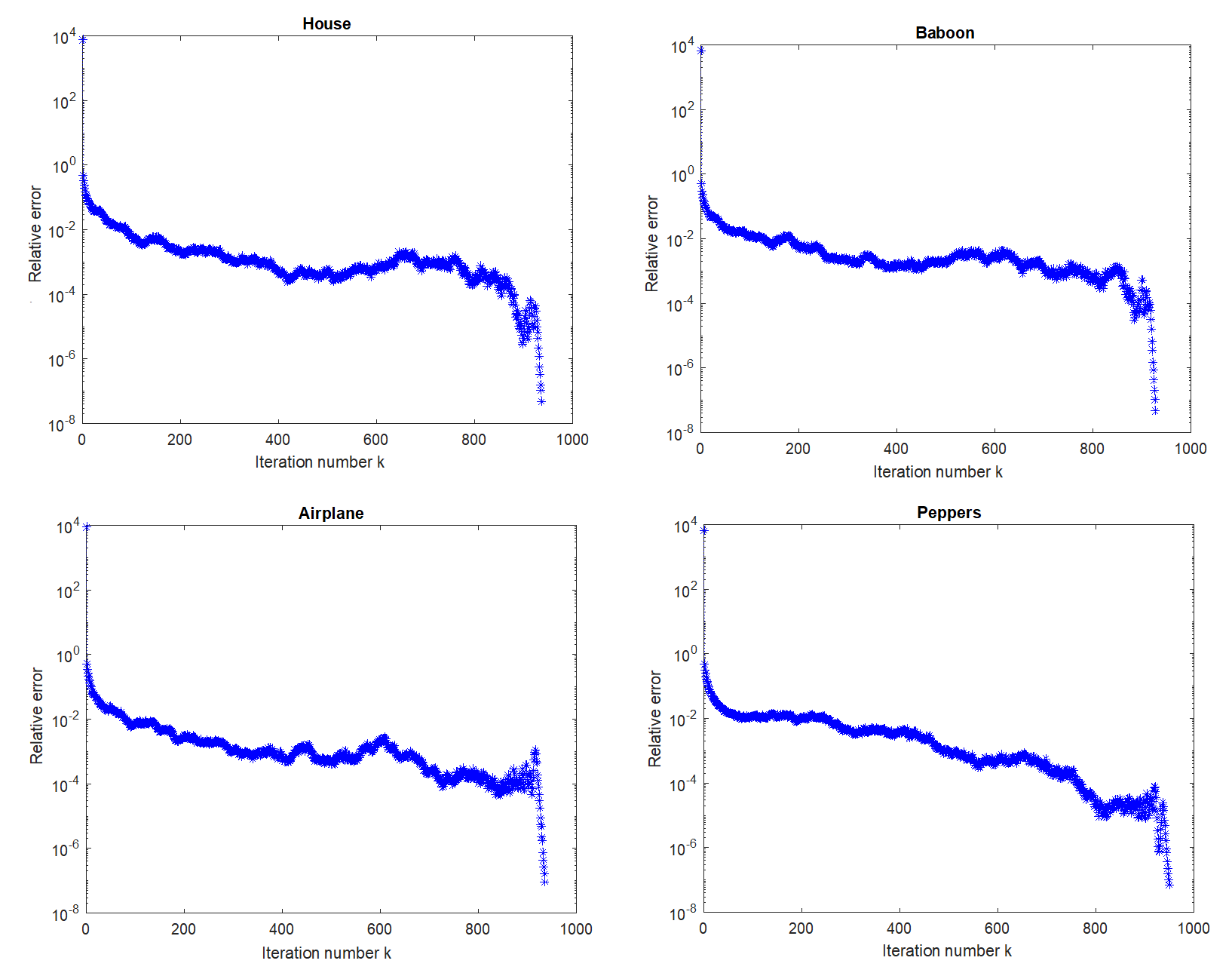}
	\caption{\small The relative error history of the proposed algorithm for four benchmark images.}\label{Err_fig}
\end{figure}

\begin{figure}[H]
	\centering
	\includegraphics[width=0.7\linewidth]{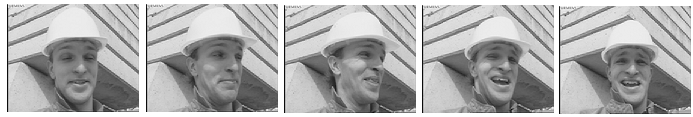}
	\includegraphics[width=0.7\linewidth]{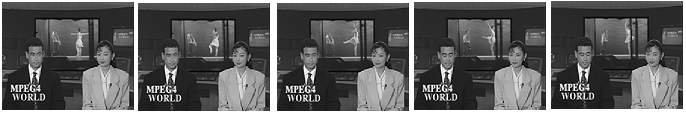}
	\caption{\small Some sample frames of ``Foreman'' (top figure) and ''News'' (bottom figure) videos.}\label{video_sample}
\end{figure}

\begin{figure}[H]
	\centering
	\includegraphics[width=0.45\linewidth]{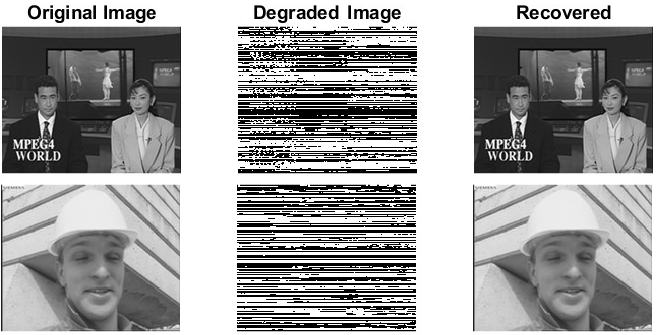}
	\caption{\small The original, the degraded and the recovered frame number 1 (as one example) using the proposed iterative method. All recovered frames had the same quality as the frame number one.}\label{video_recoverd}
\end{figure}

\begin{figure}[H]
	\centering
	\includegraphics[width=0.9\linewidth]{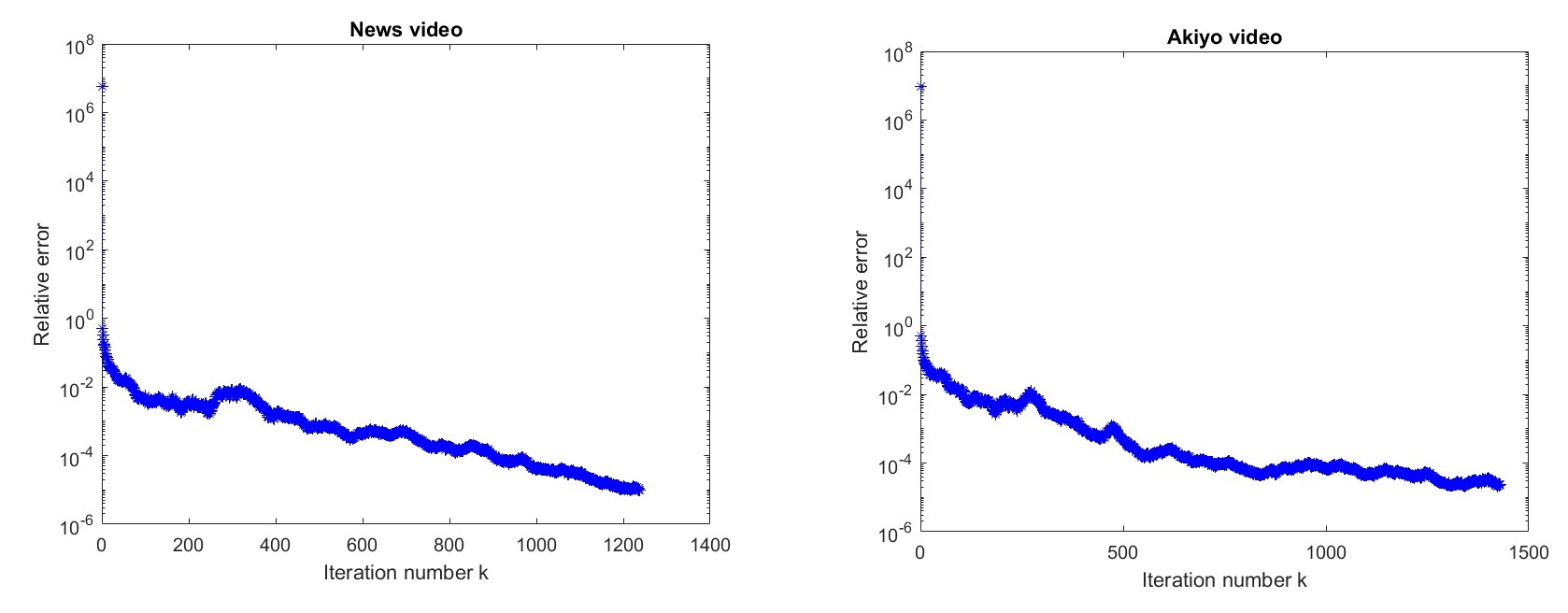}
	\caption{\small The relative error history of the proposed algorithm for two benchmark videos.}\label{Err_fig_video}
\end{figure}

\end{example}

\section{Conclusion}
This paper proposed two efficient iterative algorithms for solving tensor equations under the T-product framework. For T-symmetric positive definite equations $\mathcal{C} \star \mathcal{X} = \mathcal{D}$, a conjugate-gradient-type method was developed with guaranteed convergence in at most $n_1 n_3 + 1$ steps. For general consistent equations, a normal-equation transformation was used, and for inconsistent systems, a least-squares formulation was adopted. Rigorous finite-step convergence proofs were established, and minimal Frobenius norm solutions were obtained via appropriate initialization.

The methods operate directly on tensors, avoid explicit inversion, and achieve lower complexity, $\mathcal{O}(k n_1 n_2 n_3 \log n_3)$ versus $\mathcal{O}(n_3 n_1^3)$ for direct inversion. Numerical experiments on synthetic data, images, and videos confirmed high precision, a $2.31\times$ speedup, and $26\%$ memory usage compared to direct tubal inversion.

Future work includes preconditioning, structured tensors, parallel implementations, and extensions to other tensor products and applications.


\section*{Acknowledgement}
The authors would like to thank the editor and the anonymous reviewer for their insightful comments and suggestions, which have greatly improved the quality of this paper.
\vspace*{.5cm}



\textbf{Conflicts of Interest:} The authors declare no conflicts of interest.
\vspace*{.5cm}

\textbf{Data Availability Statement:} The data that support the findings of this study are openly available in Xiph.org at \url{https://media.xiph.org/video/derf/}.
\vspace*{.5cm}

\textbf{Funding statement:} There are no funders to report for this submission.



\begin{thebibliography}{99}

\bibitem{Ahmadi}
Ahmadi-Asl, S., An efficient randomized fixed-precision algorithm for tensor singular value decomposition. Communications on Applied Mathematics and Computation, 5(4), 1564-1583 (2023). 
\bibitem{Asl}
S. Ahmadi-Asl, M.G Asante-Mensah, A. Cichocki, A-H. Phan, I. Oseledets, J. Wang, Fast cross tensor approximation for image and video completion, Signal Processing, 213, 109121 (2023).


\bibitem{com}
Bentbib, A. H., Hachimi, A. E., Jbilou, K., Ratnani, A. A., Tensor Regularized Nuclear Norm Method for Image and Video Completion, J. Optim. Theory Appl. 192(2) 401-425 (2022).
\bibitem{fp1}
Beik, F. P., Najafi-Kalyani, M., A preconditioning technique in conjunction with Krylov subspace methods for solving multilinear systems. Appl. Math. Lett. 116 107051 (2021).
\bibitem{chang}
Chang, S. Y., Wei, Y., T product tensors part I: Inequalities. arXiv preprint arXiv:2107.06285 (2021).
\bibitem{gui2}
El Guide, M., El Ichi, A., Jbilou, K., Discrete cosine transform LSQR methods for multidimensional ill-posed problems. Journal of Mathematical Modeling, 10(1), 21-37 (2022).
\bibitem{fp}
Guide, M. E., Ichi, A. E., Beik, F. P., Jbilou, K., Tensor GMRES and Golub-Kahan Bidiagonalization methods via the Einstein product with applications to image and video processing. arXiv preprint arXiv:2005.07458 (2020).
\bibitem{gui}
El Guide, M. O. H. A. M. E. D., El Ichi, A. L. A. A., Jbilou, K., Sadaka, R., On tensor GMRES and Golub-Kahan methods via the T-product for color image processing, Electron. J. Linear Algebra. 37  524-543 (2021).
\bibitem{Elden}
Elden, L., Ahmadi-Asl, S., Solving bilinear tensor least squares problems and application to Hammerstein identification, Numer. Linear Algebra. Appl. 26 (2) (2019) e2226.
\bibitem{and}
Fusiello, A., A matter of notation: Several uses of the Kronecker product in 3D computer vision, Pattern Recognit. Lett. 28(15), 2127-2132 (2007).
\bibitem{jin}
Jin, H., Bai, M., Benítez, J., Liu, X., The generalized inverses of tensors and an application to linear models. Comput. Math. Appl. 74(3), 385-397 (2017).
\bibitem{kim3}
Kilmer, M. E., Braman, K., Hao, N., Hoover, R. C., Third-order tensors as operators on matrices: A theoretical and computational framework with applications in imaging. SIAM J. Matrix Anal. Appl. 34(1), 148-172 (2013).
\bibitem{kim2}
Kilmer, M. E., Martin, C. D., Factorization strategies for third-order tensors, Linear Algebra Appl. 435(3), 641-658 (2011).
\bibitem{9}
Khalil, N., Sarhan, A., Alshewimy, M. A., An efficient color/grayscale image encryption scheme based on hybrid chaotic maps. Opt. Laser Technol. 143, 107326 (2021).
\bibitem{svd}
Kolda, T. G., Bader, B. W., Tensor decompositions and applications. SIAM Rev. 51(3) 455-500 (2009).
\bibitem{lu}
Lu, C. Tensor-tensor product toolbox. arXiv preprint arXiv:1806.07247 (2018).
\bibitem{ana}
Lu, C., Feng, J., Chen, Y., Liu, W., Lin, Z., Yan, S., Tensor robust principal component analysis with a new tensor nuclear norm, IEEE Trans. Pattern Anal. Mach. Intell. 42(4) 925-938 (2019).
\bibitem{8}
Martin, C. D., Shafer, R., LaRue, B., An order-p tensor factorization with applications in imaging, SIAM J. Sci. Comput. 35(1) A474-A490 (2013).
\bibitem{Mensah}
MG. Asante-Mensah, S. Ahmadi-Asl, A. Cichocki, Matrix and tensor completion using tensor ring decomposition with sparse representation, Machine
Learning: Science and Technology 2(3), 035008 (2021).
\bibitem{nobakht3}
Nobakht-Kooshkghazi, M., Afshin, H., The new Krylov subspace methods for solving tensor equations via T-product. Computational and Applied Mathematics, 42(8), 358 (2023).
\bibitem{ten}
Qi, L., Luo, Z., Tensor analysis: spectral theory and special tensors, Society for Industrial and Applied Mathematics (2017).
\bibitem{sem}
Semerci, O., Hao, N., Kilmer, M. E., Miller, E. L., Tensor-based formulation and nuclear norm regularization for multienergy computed tomography, IEEE Trans. Image Process. 23(4) 1678-1693 (2014).
\bibitem{machine}
Settles, B., Craven, M., Ray, S., Multiple-instance active learning, Adv. Neural Inf. Process. Syst. 20 (2007).
\bibitem{sun1}
Sun, L., Zheng, B., Bu, C., Wei, Y.  Moore–Penrose inverse of tensors via Einstein product. Linear Multilinear Algebra. 64(4) 686-698 (2016).
\bibitem{wen}
Wang, Q. W., Xu, X., Iterative algorithms for solving some tensor equations. Linear Multilinear Algebra. 67(7), 1325-1349 (2019).
\bibitem{10}
Xu, Y., Hao, R., Yin, W., Su, Z.,  Parallel matrix factorization for low-rank tensor completion. arXiv preprint arXiv:1312.1254 (2013).
\bibitem{zhang}
Zhang, Z., Aeron, S., Exact tensor completion using t-SVD, IEEE Trans. Signal Process. 65(6) 1511-1526 (2016).
\bibitem{7}
Zhang, Z., Ely, G., Aeron, S., Hao, N., Kilmer, M., Novel methods for multilinear data completion and de-noising based on tensor-SVD, Proc. IEEE Comput. Soc. Conf. Comput. Vis. Pattern Recognit. 3842-3849 (2014).

\bibitem{LAAP}
Ahmadi-Asl, S., Phan, A. H., Cichocki, A., A randomized algorithm for tensor singular value decomposition using an arbitrary number of passes, Journal of Scientific Computing, 98(1), 23 (2024).

\bibitem{ACM}
Ahmadi-Asl, S., Phan, A. H., Caiafa, C. F., Cichocki, A., Robust low tubal rank tensor recovery using discrete empirical interpolation method with optimized slice/feature selection, Advances in Computational Mathematics, 50(2), 23 (2024).

\end{thebibliography}
\end{document}